\documentclass[a4paper,10pt]{article}
\usepackage{pst-all}
\usepackage{pst-slpe}
\usepackage{pst-poly}
\usepackage{multido}
\usepackage[dvips]{graphicx}
\usepackage{array}
\usepackage{color}
\usepackage{showlabels}

\usepackage{amssymb,amsmath,amsthm,mathrsfs}
\newtheorem{Def}{Definition}[section]
\newtheorem{Rem}{Remark}[section]

\newtheorem{Theor}{Theorem}[section]

\newtheorem{Prop}{Proposition}[section]
\newtheorem{Lem}{Lemma}[section]
\newtheorem{assumption}{Assumption}[section]
\theoremstyle{plain}
\DeclareMathOperator{\di}{div}

\DeclareMathOperator{\inter}{int}

\newcommand{\ve}{\varepsilon}
\newcommand{\na}{\nabla}
\newcommand{\non}{\nonumber}
\newcommand{\cR}{\mathbb R}
\newcommand{\eq}[1]{\mbox{\rm {(\ref{#1})}}}

\newcommand{\yieldlimit}{\sigma_{\mathrm{y}}}
\DeclareMathOperator{\dom}{dom}
\DeclareMathOperator{\epi}{epi}

\global\long\def\eps{\varepsilon}

\global\long\def\cB{{\cal B}}

\global\long\def\cF{{\cal F}}

\global\long\def\cM{{\cal M}}
\global\long\def\cP{{\cal P}}
\global\long\def\cR{{\cal R}}

\global\long\def\norm#1{\left\Vert #1\right\Vert }

\global\long\def\weakto{\rightharpoonup}

\global\long\def\intomega{\int_{\Omega}}
\global\long\def\intrn{\int_{\Rn}}

\global\long\def\muomega{\mu_{\omega}}

\global\long\def\mupalm{\mu_{\cP}}

\global\long\def\tsweak{\stackrel{{\scriptstyle 2s}}{\weakto}}

\global\long\def\lebesgueL{{\cal L}}

\global\long\def\sL{\mathscr{L}}

\global\long\def\R{\mathbb{R}}
\global\long\def\Rn{\mathbb{R}^{n}}
\global\long\def\Rnn{\mathbb{R}^{n\times n}}

\global\long\def\Rd{\mathbb{R}^{d}}

\global\long\def\N{\mathbb{N}}

\global\long\def\v{\upsilon}

\global\long\def\bQ{\boldsymbol{Q}}

\global\long\def\ueta{u^{\eta}}
\global\long\def\ue{u^{\eps}}

\global\long\def\d{{\rm d}}

\begin{document}

\title{Stochastic homogenization of rate-dependent models of monotone type in plasticity}


\author{Martin Heida%
\thanks{Martin Heida, Weierstrass Institute for Applied Analysis and Stochastics, 
Mohrenstrasse 39, 10117 Berlin, Germany, email: Martin.Heida@wias-berlin.de, 
Tel.: +49 30 203-72-562}\,
, Sergiy Nesenenko%
\thanks{Corresponding author: Sergiy Nesenenko, Fakult\"at II, Institut f\"ur Mathematik, Technische 
Universit\"at Berlin, Strasse des 17. Juni 136, 10623 Berlin, Germany, 
email: sergiy.nesenenko@math.tu-berlin.de, Tel.: +49 (0)30 314-29-267}
}

\date{\today}
\maketitle
\begin{abstract}
In this work we deal with the stochastic homogenization of the initial
boundary value problems of monotone type. The models of monotone type under
consideration describe 
the deformation behaviour of inelastic materials with a microstructure which can be
characterised by random measures.  Based on the Fitzpatrick function concept we reduce 
the study of the asymptotic behaviour of monotone operators associated with our models
 to the problem of the stochastic 
homogenization of convex functionals within an ergodic and stationary setting.
The concept of Fitzpatrick's function helps us to introduce and show the existence of the
 weak solutions for rate-dependent systems. 
 The derivations of the homogenization results presented in this work are
based on the stochastic two-scale convergence in Sobolev spaces. For completeness, we also present some two-scale homogenization results for convex functionals, which are related to the classical $\Gamma$-convergence theory. 
\end{abstract}

\\[2ex]

\section{Introduction}
In this work we are concerned with the homogenization of the initial
boundary value problem describing the deformation behavior of
inelastic materials with a microstructure which can be characterised by random measures.   

While the periodic homogenization theory
for elasto/visco-plastic models  is sufficiently well established
(see \cite{AlbNese09b,Francfort_Giacomini_2015,Miel07,Nes07,Nesenenko12a,Schweizer10,Schweizer_Veneroni_2014,Vis08b,Visintin08} and references therein), 
some improvement in the development of the techniques for the stochastic homogenization of 
the quasi-static initial boundary value problems of monotone type has to be achieved yet. 
To the best knowledge of the authors, there are only two works (\cite{heida2014stochastic,heida2016a}) 
available on the market which are concerned with the homogenization problem of 
rate-independent systems in plasticity within an ergodic and stationary setting. In this work
we extend the results obtained in \cite{heida2016a} for perfectly elasto-plastic models to rate-dependent plasticity. Our main ingredient in the construction of the stochastic homogenization theory
for rate-dependent models of monotone type is the combination of the Fitzpatrick function concept and the two-scale convergence technique in spaces equipped with random measures due to V.V. Zhikov and A.L. Pyatnitskii (see \cite{Zhikov_Pyatnitskii_2006}).
The Fitzpatrick function is used here to reduce 
the study of the asymptotic behaviour of monotone operators associated with the models
under consideration to the problem of the stochastic 
homogenization of convex functionals defined on Sobolev spaces with random measures.
\paragraph{Setting of the problem.}
 Let $\bQ\subset\mathbb{R}^3$ be an open bounded set, the set of material
points of the solid body, with a Lipschitz boundary $\partial\bQ$, 
the number $\eta >0$  denote the scaling
parameter of the microstructure and $T_e$ be some positive number (time of existence).
For $0 < t\leq T_e$
\[\bQ_t = \bQ \times (0, t). \]
 Let ${\cal S}^3$ denote the set of symmetric $3 \times
3$-matrices, and let $u_\eta(x,t) \in {\mathbb R}^3$ be the unknown
 displacement of the material point $x$ at time $t$,
$\sigma_\eta(x,t) \in {\cal S}^3 $ be the unknown Cauchy stress tensor and
$z_\eta(x,t) \in {\mathbb R}^N$ denote the unknown
 vector of internal variables.
The model equations of the problem (the microscopic problem) are
\begin{eqnarray}
\label{MPr1}  - \di_x \sigma_\eta(x,t) &=&  b(x,t),  \\
\label{MPr2} \sigma_\eta(x,t) &=&  {\mathbb C}_\eta[x](\varepsilon(\na_x u_\eta(x,t))-B z_\eta(x,t)),\\
\label{MPr3} \partial_t z_\eta(x,t) & \in &
g_\eta\big(x, B^T \sigma_\eta(x,t) - L_\eta[x]z_\eta(x,t)\big),
\end{eqnarray} 
together with the homogeneous Dirichlet boundary condition
\begin{eqnarray}
\label{MPr4}  u_\eta(x,t)=0, \hspace{5ex} (x,t)\in \partial\bQ\times(0,\infty),
\end{eqnarray} 
and the initial condition
\begin{eqnarray}
\label{MPr5}  z_\eta(x,0)=z_\eta^{(0)}(x), \hspace{12ex} x\in\bQ.
\end{eqnarray} 
In model equations \eq{MPr1} - \eq{MPr5}
 $$\varepsilon (\na_x u_\eta(x, t)) = \frac{1}{2} ( \nabla _{x} u_\eta(x,
t) + ( \nabla _{x} u_\eta(x, t))^T ) \in {\cal S}^3$$
denotes the strain
tensor (the measure of deformation), $B: {\mathbb R}^N \rightarrow {\cal S}^3$ is a linear mapping,
which assigns to each vector of internal variables $z_\eta(x, t)$ the plastic strain tensor
${\varepsilon}_{p,\eta} (x,t)\in {\cal S}^3$, i.e. the following relation
${\varepsilon}_{p,\eta} (x,t)=Bz_\eta(x,t)$ holds. We recall that the space ${\cal S}^3$ can be isomorphically identified 
 with the space ${\mathbb R}^6$ (see \cite[p. 31]{Alb98}). Therefore, the linear mapping $B: {\mathbb R}^N\to{\cal S}^3$ is defined as a composition
of a projector from ${\mathbb R}^N$ onto ${\mathbb R}^6$ and the isomorphism
between ${\mathbb R}^6$ and ${\cal S}^3$.
The transpose 
$B^T: {\cal S}^3 \rightarrow {\mathbb R}^N$ is given by
$$B^T{v}=(\hat{z},0)^T$$
for $v\in{\cal S}^3$ and $z=(\hat{z}, \tilde{z})^T\in{\mathbb R}^N,$ $\hat{z}\in{\mathbb R}^6,$
$\tilde{z}\in{\mathbb R}^{N-6}$.

For every $x \in\bQ$ we denote
by ${\mathbb C}_\eta[x]: {\cal S}^3 \rightarrow {\cal S}^3$ a linear
symmetric mapping, the elasticity tensor. It is assumed that the mapping
  $x \rightarrow {\mathbb C}_\eta[x]$ is measurable. Further, we suppose that there
exist two positive constants $0 < \alpha <\beta$ such that the two-sided inequality
\begin{eqnarray}
    \alpha | \xi | ^{2} \leq {\mathbb C}_\eta[ x ] \xi \cdot \xi 
    \leq \beta | \xi | ^{2} \ \ \ \textrm{for} \ \textrm{any} \ \xi
    \in {\cal S}^3. \non
\end{eqnarray}
is satisfied uniformly with respect to $x\in\bQ$ and $\eta>0$.
The given function $b: \bQ\times [0, \infty ) \rightarrow {\mathbb R}^3$ is
the volume force.
The ($N\times N$)-matrix $L_\eta[x]$ represents hardening effects. It is assumed to be
positive semi-definite, only.
For all $x \in\bQ$ the function $z \rightarrow g_\eta(x,z): {\mathbb R}^N
\rightarrow 2^{{\mathbb R}^N}$ is maximal monotone and
satisfies the following condition $$0 \in g_\eta(x, 0), \hspace{4ex}x\in\bQ.$$
 The mapping $x \rightarrow \left(L_\eta[x], g_\eta(x,\cdot)\right)$ is measurable. 
  \begin{Rem} Visco-plasticity is typically included in the former conditions by choosing the function $g_\eta$ to be in Norton-Hoff form, i.e. 
\begin{align*}
  g_\eta(x,\Sigma)=[|{\Sigma}|-\yieldlimit(x)]_+^{r_\eta(x)}\,\frac{\Sigma}{|\Sigma|}\, , 
  \quad\quad\Sigma\in{\cal S}^3,\, x\in\bQ,
\end{align*}  
where $\yieldlimit:\bQ\to(0, \infty)$ is the flow stress function and $r_\eta:\bQ\to(0, \infty)$ is some material function together with $[x]_+:=\max(x,0)$.  
\end{Rem}
In order to specify the dependence of the model coefficients in \eq{MPr1} - \eq{MPr5} 
on the microstructure scaling parameter $\eta>0$, we introduce the concept of a spatial
dynamical system. Throughout this paper, we follow the setting of Papanicolaou and Varadhan
\cite{papanicolaou1979boundary} and make the following assumptions.
\begin{assumption}
\label{assu:Omega-mu-tau}Let $(\Omega,\cF_{\Omega},\cP)$ be a probability
space with countably generated $\sigma$-algebra $\cF_{\Omega}$.
Further, we assume we are given a family $(\tau_{x})_{x\in\Rn}$ of
measurable bijective mappings $\tau_{x}:\Omega\mapsto\Omega$, having
the properties of a \emph{dynamical system }on $(\Omega,\cF_{\Omega},\cP)$,
i.e. they satisfy (i)-(iii):
\begin{enumerate}
\item [(i)]$\tau_{x}\circ\tau_{y}=\tau_{x+y}$ , $\tau_{0}=id$ (Group
property)
\item [(ii)]\label{enu:measure-preserv}$\cP(\tau_{-x}B)=\cP(B)\quad\forall x\in\Rn,\,\,B\in\cF_{\Omega}$
(Measure preserving)
\item [(iii)]$A:\Rn\times\Omega\rightarrow\Omega,\,(x,\omega)\mapsto\tau_{x}\omega$
is measurable (Measurablility of evaluation) 
\end{enumerate}
We finally assume that the system $(\tau_{x})_{x\in\Rn}$ is ergodic.
This means that for every measurable function $f:\Omega\rightarrow\R$
there holds 
\begin{equation}
\begin{split}\left[f(\omega)\stackrel{}{=}f(\tau_{x}\omega)\,\,\forall x\in\Rn\,,\,a.e.\,\,\omega\in\Omega\right] & \Rightarrow\left[f(\omega)=const\,\,\cP\textnormal{-a.e.}\,\omega\in\Omega\right]\,.\end{split}
\label{eq:def_ergodicity}
\end{equation}
\end{assumption}
For reader's convenience, we recall the following well-known result (see \cite[Section VI.15]{Doob1994}). 
\begin{Lem}
Let $(A,\cF,\mu)$ be a finite measure space with countably generated
$\sigma$-algebra $\cF$. Then, for all $1\leq p<\infty$, $L^{p}(A;\mu)$
contains a countable dense set of simple functions. 
\end{Lem}
The coefficients in \eq{MPr1} - \eq{MPr5} are defined as follows. First, we define the stationary
random fields through the relations
\begin{eqnarray}
{\mathbb C}[x, \omega]=\tilde{\mathbb C}[\tau_{x}\omega], \hspace{2ex}L[x, \omega]=\tilde{L}[\tau_{x}\omega],\non
\end{eqnarray}
and for every fixed $v\in{\mathbb R}^N$
\begin{eqnarray}
g(x, \omega, v)=\tilde{g}(\tau_{x}\omega,v),\non
\end{eqnarray}
where $\tilde{\mathbb C}$, $\tilde{L}$ 
are measurable functions over $\Omega$ and 
$\omega\mapsto \tilde{g}(\omega,\cdot)$  is 
measurable in the sense of Definition \ref{def:meas-max-monotone}. Then,
given the specified assumptions on the random fields, the coefficients 
${\mathbb C}_\eta[x]$, $L_\eta[x]$ and the mapping
$x\mapsto g_\eta(x,\cdot)$ are defined as
\begin{eqnarray}
{\mathbb C}_\eta[x]={\mathbb C}\left[\frac{x}{\eta}, \omega\right], 
\hspace{2ex}L_\eta[x]=L\left[\frac{x}{\eta}, \omega\right],\non
\end{eqnarray}
and for each fixed $v\in{\mathbb R}^N$
\begin{eqnarray}
g_\eta(x,v)=g\left(\frac{x}{\eta}, \omega, v\right).\non
\end{eqnarray}
Furthermore, we assume that 
 \[z^{(0)}_\eta(x)=\tilde{z}^{(0)}\left(x, \tau_{\frac{x}{\eta}}\omega\right),\hspace{2ex}x\in\bQ.\]
for some ergodic function $\tilde{z}^{(0)}\in L^2(\bQ\times\Omega; {\cal L}\otimes\mu)$.

From a modelling perspective, this construction is equivalent  to the assumption that 
the coefficients and the given functions in \eq{MPr1} - \eq{MPr5} are statistically homogeneous
(see \cite{Daley1988}, for example).
\paragraph{Notation.} The symbols $|\cdot|$ and $(\cdot,\cdot)$ will denote a norm and
a scalar product in ${\mathbb R}^k$, respectively.
Let $S$ be a measurable set in ${\mathbb R}^s$. For $m\in \mathbb N$, $q\in [1,\infty]$,
 we denote by $W^{m,q}(S, {\mathbb R}^k)$ the Banach space of Lebesgue
integrable functions having $q$-integrable  weak derivatives up to  order
$m$. This space is equipped with the norm $\| \cdot \|_{m,q,S}$. If $m=0$,
we write $\| \cdot \|_{q,S}$; and if (additionally) $q=2$,
we also write $\| \cdot \|_{S}$.  
We set $H^m(S, {\mathbb R}^k)= W^{m,2}(S, {\mathbb R}^k)$.
We choose the numbers $p, q$ satisfying $1 < p, q < \infty$ and
$1/p + 1/q = 1$. 
For such $p$ and $q$ one can define the bilinear 
form on the product space
$L^{p}(S, {\mathbb R}^k)\times L^{q}(S, {\mathbb R}^k)$ by
\[
(\xi, \zeta )_{S} = \int_{S} (\xi(s), \zeta(s))ds.
\]
For
functions $v$ defined on $\Omega \times [0,\infty)$ we denote by
$v(t)$ the mapping $x \mapsto v(x,t)$, which is defined on $\Omega$.
 The space $L^q(0,T_e; X)$ denotes the Banach space of all Bochner-measurable 
functions $u:[0,T_e)\to X$ such that $t\mapsto\|u(t)\|^q_X$ is integrable
on $[0,T_e)$. Finally, we frequently use the spaces $W^{m,q}(0,T_e;X)$, 
which consist of Bochner measurable functions having $q$-integrable weak
derivatives up to order $m$.


\section{Preliminaries.}\label{BasicsConAna}

In this section we briefly recall some basic facts from convex analysis and 
nonlinear functional analysis which are needed for further discussions. For more details see 
\cite{Barb76,Hu97,Pas78,Zalinescu02}, for example.
 
Let $V$ be a reflexive Banach space with the norm $\|\cdot\|$, $V^*$ 
be its dual space with the norm $\|\cdot\|_*$. The
brackets $\left< \cdot ,\cdot \right>$ denote the duality pairing between
$V$ and $V^*$. By $V$ we shall always mean a reflexive Banach space
throughout this section.

For a function $\phi:V \to \overline{\mathbb R}$ the sets
\[\dom (\phi)=\{v\in V\mid \phi(v)<\infty\}, \ \
\epi(\phi)=\{(v,t)\in V\times \cR\mid \phi(v)\le t\}\]
are called the {\it effective domain} and the {\it epigraph} 
of $\phi$, respectively. One says that the function $\phi$ is {\it proper}
if $\dom(\phi)\not=\emptyset$ and $\phi(v)>-\infty$ for every $v\in V$.
The epigraph is a non-empty closed convex
set iff $\phi$ is a proper lower semi-continuous convex function or,
equivalently, iff $\phi$ is a proper weakly
 lower semi-continuous convex function 
(see \cite[Theorem 2.2.1]{Zalinescu02}).

 The Legendre-Fenchel
conjugate of a proper convex lower semi-continuous function $\phi : V \to 
\overline{\mathbb R}$ is the function $\phi^*$ defined for each 
$v^* \in V^*$ by
\[\phi^*(v^*)=\sup_{v \in V} \{ \left< v^*, v \right> - \phi(v)\}. \]
The Legendre-Fenchel conjugate $\phi^*$ is convex, lower semi-continuous
and proper on the dual space $ V^*$. Moreover, the 
{\it Young-Fenchel inequality} holds
\begin{eqnarray}\label{YoungFenchelIneq}
\forall v\in V,\ \forall v^*\in V^*:\ \ \phi^*(v^*)+\phi(v)\ge 
\left< v^*, v \right>,
\end{eqnarray}
and the inequality $\phi\le\psi$ implies $\psi^*\le\phi^*$ for any two
 proper convex lower semi-continuous functions $\psi,\phi: V \to 
\overline{\mathbb R}$ (see \cite[Theorem 2.3.1]{Zalinescu02}). 

Due to Proposition II.2.5 in \cite{Barb76} a proper convex 
lower semi-continuous function $\phi$ satisfies the following identity
\begin{eqnarray}\label{domainConvFunc}
\inter\dom(\phi)=\inter\dom(\partial\phi),
\end{eqnarray}
where $\partial\phi: V \to 2^{V^*}$ denotes the subdifferential 
of the function $\phi$. We note that the equality in (\ref{YoungFenchelIneq})
holds iff $v^*\in\partial\phi(v)$.
\begin{Rem}\label{SubMax} We recall that the subdifferential of
 a lower semi-continuous proper and
convex function is maximal monotone
 (see \cite[Theorem II.2.1]{Barb76}) in the sense of Definition \ref{def:monotone} below.
\end{Rem}
 \paragraph{Convex integrands.} Let the numbers $p, q$ satisfy
  $1 < q\le 2\le p < \infty, 1/p + 1/q = 1.$
For a proper convex lower semi-continuous function 
$\phi:{\mathbb R}^k\to\overline{\mathbb R}$ we define a functional $I_{\phi}$ on
$L^p(G,{\mathbb R}^k)$ by
\[I_{\phi}(v)=\begin{cases}\int_G\phi(v(x))dx, & \phi(v)\in 
                                      L^1(G,{\mathbb R}^k)\\
                       +\infty, & {\rm otherwise}
  \end{cases},\]
where $G$ is a bounded domain in ${\mathbb R}^N$ with some
$N\in{\mathbb N}$. Due to
Proposition II.8.1 in \cite{Show97}, the functional $I_{\phi}$
is proper, convex, lower semi-continuous, and $v^*\in\partial I_{\phi}(v)$ iff
\begin{eqnarray}
v^*\in L^{q}(G,{\mathbb R}^k), \ 
\ v\in L^p(G,{\mathbb R}^k)\ \ {\rm and} \ \ 
\ v^*(x)\in\partial{\phi}(v(x)), \ {\rm a.e.} \non
\end{eqnarray} 
Due to the result of Rockafellar in \cite[Theorem 2]{Rockafellar68},
the Legendre-Fenchel conjugate of $I_{\phi}$ is equal to $I_{\phi^*}$, i.e.
\[\big(I_{\phi}\big)^*=I_{\phi^*},\]
where $\phi^*$ is the Legendre-Fenchel conjugate of $\phi$.
\paragraph{Maximal monotone operators.}
For a multivalued mapping $A:V \to 2^{V^*}$ the sets
\[D(A)=\{v\in V\mid Av\not=\emptyset\},\hspace{2ex} Gr A=\{[v,v^*]\in V\times V^*\mid v\in D(A),\ v^*\in Av\}\]
are called the {\it effective domain} and the {\it graph} of $A$, respectively.
\begin{Def}\label{def:monotone}
A mapping $A:V \to 2^{V^*}$ is called {\rm monotone} if and only if the following
inequality holds
 \[\left< v^* - u^*, v - u \right> \ge 0  \ \ \ \ \forall \ [v,v^*],
 [u,u^*]\in Gr A.\]
 
A monotone mapping $A:V \to 2^{V^*}$ is called {\rm maximal monotone} iff the
inequality
 \[\left< v^* - u^*, v - u \right> \ge 0  \ \ \ \ \forall \ [u,u^*]\in Gr A\]
implies $[v,v^*]\in Gr A$.
\end{Def}
It is well known (\cite[p. 105]{Pas78}) that if $A$ is a maximal monotone
operator, then for any $v\in D(A)$
the image $Av$ is a closed convex subset of $V^*$ and the graph
$Gr A$ is demi-closed\footnote{A set $A\in V\times V^*$ is demi-closed
if $v_n$ converges strongly to $v_0$ in $V$ and $v^*_n$ converges weakly
to $v^*_0$ in $V^*$ (or $v_n$ converges weakly to $v_0$ in $V$ and $v^*_n$ converges strongly to $v^*_0$ in $V^*$) and $[v_n,v_n^*]\in Gr A$,
 then $[v,v^*]\in Gr A$}.
\paragraph{Canonical extensions of maximal monotone operators.}\label{measurabilityMultis}
In this subsection we briefly present some facts about measurable
multi-valued mappings (see \cite{AubinFrankowska2008,Castaing77,Hu97,Pankov97}, for example).
We assume that $V$, and hence $V^*$, is separable and denote 
the set of maximal monotone operators from $V$ to $V^*$
 by ${\mathfrak M}(V \times V^*)$. Further, let 
$(S, \Sigma(S), \mu)$ be a $\sigma-$finite $\mu-$complete
 measurable space. 
\begin{Def} \label{def:meas-max-monotone}
A mapping $A:S\to{\mathfrak M}(V \times V^*)$ is measurable 
iff for every open set $U\in V \times V^*$ 
(respectively closed set, Borel set, open ball, closed ball),
$$\{x \in S\mid A(x)\cap U \not= \emptyset\}$$
is measurable in $S$.
\end{Def}
The fact that the closed or Borel sets can be equivalently used
 in Definition~\ref{def:meas-max-monotone} follows from
the closedness of the values of the mapping $A:S\to{\mathfrak M}(V \times V^*)$ 
(see \cite[Theorem 8.1.4]{AubinFrankowska2008}).
\begin{Rem}\label{RemMeasurabilityMultis}
Theorem 8.1.4 in \cite{AubinFrankowska2008} also implies that under the above conditions
the measurability of a mapping  $A:S\to{\mathfrak M}(V \times V^*)$ is equivalent to the existence
of a countable dense subset consisting of measurable selectors, i.e. there exists a sequence of measurable 
functions $\{v_n\}_{n\in\mathbb{N}}:\,S\to V\times V^*$ such that for any $x\in S$ the image $A(x)$
can be represented as follows
\[A(x)=\overline{\cup_{n\in\mathbb{N}} v_n(x)}.\]
\end{Rem}
 The following lemma will be used in the sequel
(see \cite[Lemma 3.1]{Visintin2013b}).
\begin{Lem}
Let a mapping $A:S\to{\mathfrak M}(V \times V^*)$ be measurable. For any ${\cal L}(S)$-measurable function
$v:S \to V$, the multivalued mapping $\hat A: x\mapsto A(x, v(x))$ is then closed-valued and measurable.
\end{Lem} 
Given a mapping $A:S\to{\mathfrak M}(V \times V^*)$, one can define
 a monotone graph from $L^p(S,V)$ to $L^q(S,V^*)$, where
$1/p + 1/q = 1$, as follows:
\begin{Def}\label{CanExtension}
 Let $A:S\to{\mathfrak M}(V \times V^*)$. The canonical extension of $A$
 from $L^p(S,V)$ to $L^q(S,V^*)$, where
$1/p + 1/q = 1$, is defined by:
\[Gr {\cal A}_p = \{[v, v^*] \in L^p(S,V)\times L^q(S,V^*)\mid
 [v(x), v^*(x)] \in Gr A(x)\ for\ a.e.\ x\in S\}.\]
\end{Def}
In the following, we will drop the index $p$ for readability. Since we always work fix $p$ at the beginning of a statement, there cannot occur confusion with this notation. Monotonicity of ${\cal A}$ defined in Definition~\ref{CanExtension}
 is obvious, while its maximality follows from the next proposition (see \cite[Proposition 2.13]{Damlamian07}).
\begin{Prop}
 Let $A:S\to{\mathfrak M}(V \times V^*)$ be measurable. 
If $Gr {\cal A}\not= \emptyset$, then ${\cal A}$ is maximal monotone.
\end{Prop}
\begin{Rem}
We point out that the maximality of $A(x)$ for almost every $x\in S$ does not imply
the maximality of ${\cal A}$ as the latter can be empty (see \cite{Damlamian07}).
\end{Rem}
\paragraph{Fitzpatrick's function.} 
  For a proper operator $\beta: V\to 2^{V^{*}}$
the Fitzpatrick function is defined as the convex and lower semicontinuous function given by
\begin{eqnarray}
f_\beta(v, v^*)=\sup\{\left<v^*,v_0\right>-\left<v^*_0,v_0-v\right>\mid v^*_0\in\beta(v_0)\}, \hspace{3ex} \forall(v,v^*)\in V\times V^*. \label{Fitzpatrick_function}
\end{eqnarray}
It is known (\cite{Fitzpatrick1988}) that, whenever $\beta$ is maximal monotone,
\begin{eqnarray}
&&f_\beta(v, v^*)\ge\left<v^*,v\right>, \hspace{3ex} \forall(v,v^*)\in V\times V^*,
 \label{Fitzpatrick_relation1}\\
&&f_\beta(v, v^*)=\left<v^*,v\right> \hspace{3ex}\Leftrightarrow \hspace{3ex}v^*\in\beta(v). \label{Fitzpatrick_relation2}
\end{eqnarray}
Any measurable maximal monotone operator $A:S\to{\mathfrak M}(V \times V^*)$ can be represented
by its Fitzpatrick function $f_A:S\times V \times V^* \to\overline{\mathbb R}$, which is 
$\Sigma(S)\otimes{\cal B}(V \times V^*)$-measurable. Namely, the graph of a mapping $A:S\to{\mathfrak M}(V \times V^*)$
can be written in the form (see \cite[Proposition 3.2]{Visintin2013b})
\[\text{for}\ a.e.\ x\in S\,\, Gr A(x)= \left\{[v, v^*] \in V\times V^*\mid f_A(x, v, v^*)=\left< v, v^* \right>\right\}.\]
We note that the measurability of the Fitzpatrick function $f_A:S\times V \times V^* \to\overline{\mathbb R}$ follows directly from its definition and
Remark~\ref{RemMeasurabilityMultis}.

The graph of the canonical extension of a measurable operator 
$A:S\to{\mathfrak M}(V \times V^*)$ can be equivalently represented in terms of its 
Fitzpatrick function $F_{{\cal A}_p}: L^p(S,V)\times L^q(S,V^*) \to\overline{\mathbb R}$, i.e.
\[Gr {\cal A}_p= \left\{[v, v^*] \in L^p(S,V)\times L^q(S,V^*)\mid F_{{\cal A}_p}(v, v^*)=\left< v, v^* \right>\right\}.\]
Again, we omit $p$ if no confusion occurs. Moreover, the following result holds (see \cite[Proposition 3.3]{Visintin2013b})
\begin{itemize}
\item the functional $F_{\cal A}$ is convex and lower semi-continuous;
\item for any $[v, v^*] \in L^p(S,V)\times L^q(S,V^*)$, the integral
\[F_{\cal A}(v, v^*)=\int_Sf_A(x, v(x), v^*(x))dx\]
exists either finite or equal to $+\infty$;
\item if there exists a pair $[v, v^*] \in L^p(S,V)\times L^q(S,V^*)$ such that $F_{\cal A}(v, v^*)<+\infty$, then
\[F^*_{\cal A}(v^*, v)=\int_Sf^*_A(x, v^*(x), v(x))dx\]
holds for all $[v^*, v] \in L^q(S,V^*)\times L^p(S,V)$.
\end{itemize}





\section{Existence of solutions}\label{Existence}

In this section we introduce and show the existence of weak solutions for
the initial boundary value \eq{MPr1} - \eq{MPr5}. To simplify the notations, throughout the whole section we ignore the fact the coefficients and the given functions in \eq{MPr1} - \eq{MPr5} depend
on $\omega\in\Omega$. The results proved below hold for a.e. $\omega\in\Omega$.
\paragraph{Solvability concept.} We start this section with the presentation of 
the intuitive ideas which lead to the definition of weak solutions for the initial boundary 
value problem  (\ref{MPr1}) - (\ref{MPr5}). To give a meaning for the solvability of
 problem (\ref{MPr1}) - (\ref{MPr5}) we are going to use
 the concept of Fitzpatrick functions defined in \eq {Fitzpatrick_function}.
 
 We assume first that a triple of functions $(u_\eta,\sigma_\eta,z_\eta)$ is given with the following properties:
for every $t\in (0,T_e)$
the function $(u_\eta(t),\sigma_\eta(t))$ is a weak solution of the boundary value problem 
\begin{eqnarray}
- \di_x \sigma_\eta(x,t) &=&  b(x,t), \label{LePr1}\\
\sigma_\eta(x,t) &=& {\mathbb C}_\eta[x](\varepsilon(\na_xu_\eta(x,t)) - Bz_\eta(x,t)), 
  \label{LePr2}\\
u_\eta(x,t) &=& 0, \hspace{6ex} x \in \partial {\bQ}.\label{LePr3}
\end{eqnarray}
This particularly holds for $z_\eta(0)=z_\eta^{(0)}$ and the 
corresponding initial values 
$(u_\eta(0),\sigma_\eta(0))=(u_\eta^{(0)},\sigma_\eta^{(0)})$. 
The equations (\ref{MPr3}) - (\ref{MPr5}) are satisfied pointwise
  for almost every $(x,t)$, and $b$ as well as $(u_\eta,\sigma_\eta,z_\eta)$ are smooth enough.
 Then, based on equivalence
\eq{Fitzpatrick_relation2}, we can rewrite equation \eq{MPr3} as follows
\[f_{g_\eta}\left(x, B^T\sigma_\eta(x,t)-L_\eta[x]z_\eta(x,t), \partial_tz_\eta(x,t)\right)\]\[
=\left(B^T\sigma_\eta(x,t)-L_\eta[x]z_\eta(x,t), \partial_tz_\eta(x,t)\right),\]
which holds for almost every  $(x,t)\in {\bQ}\times(0, T_e)$. Integrating the last equality over ${\bQ}$ gives
 \begin{eqnarray}\label{Eqiv_Fitzpatrick_Eq_1}
 \int_{\bQ}f_{g_\eta}\left(x, B^T\sigma_\eta-L_\eta z_\eta, \partial_tz_\eta\right)dx
=\int_{\bQ}\left(B^T\sigma_\eta-L_\eta z_\eta, \partial_tz_\eta\right)dx.
\end{eqnarray}
Using (\ref{MPr1}), (\ref{MPr2}) and (\ref{MPr4}) the right hand side in \eq{Eqiv_Fitzpatrick_Eq_1}
becomes (${\mathbb A}_\eta:={\mathbb C}_\eta^{-1}$)
 \begin{eqnarray}\label{Eqiv_Fitzpatrick_Eq_2}
 \int_{\bQ}\left(B^T\sigma_\eta-L_\eta z_\eta, \partial_tz_\eta\right)dx= 
 \left(B^T\sigma_\eta, \partial_tz_\eta\right)_{\bQ}-\frac12\frac{d}{dt}\left\|L^{1/2}_\eta z_\eta\right\|_{\bQ}^2
 \non\\
 = \left(\sigma_\eta, \ve(\partial_t\na_xu_\eta)\right)_{\bQ}-
 \left({\mathbb A}_\eta\sigma_\eta, \partial_t\sigma_\eta\right)_{\bQ}-\frac12\frac{d}{dt}\left\|L^{1/2}_\eta z_\eta\right\|_{\bQ}^2
 \non\\
 = \left(b, \partial_tu_\eta\right)_{\bQ}
-\frac12\frac{d}{dt}\left\{\left\|{\mathbb A}^{1/2}_\eta \sigma_\eta\right\|_{\bQ}^2+\left\|L^{1/2}_\eta z_\eta\right\|_{\bQ}^2\right\}.
\end{eqnarray}
Integrating relations \eq{Eqiv_Fitzpatrick_Eq_1} and \eq{Eqiv_Fitzpatrick_Eq_2} with respect to $t$
leads to
 \begin{eqnarray}\label{Eqiv_Fitzpatrick_Eq_3}
 &&\int_{\bQ} \left({\mathbb A}_\eta[x]\sigma_\eta(x,t), \sigma_\eta(x,t)\right)dx+
\int_{\bQ} \left(L_\eta[x]z_\eta(x,t), z_\eta(x,t)\right)dx
 \non\\
&&+\int_0^t\int_{\bQ}f_{g_\eta}\left(x, B^T\sigma_\eta(x,\tau)-L_\eta[x]z_\eta(x,\tau), \partial_\tau z_\eta(x,\tau)\right)dxd\tau\\
&& =\int_{\bQ} \left({\mathbb A}_\eta[x]\sigma_\eta(x,0), \sigma_\eta(x,0)\right)dx+
\int_{\bQ} \left(L_\eta[x]z^{(0)}_\eta(x), z^{(0)}_\eta(x)\right)dx+ \left(b, \partial_\tau u_\eta\right)_{{\bQ}_t}.\non
\end{eqnarray}
Taking into account the inequality \eq{Fitzpatrick_relation1}, we conclude that the triple
of functions $(u_\eta,\sigma_\eta,z_\eta)$ satisfies
equality \eq{Eqiv_Fitzpatrick_Eq_3} if and only if the inequality
\begin{eqnarray}\label{Eqiv_Fitzpatrick_Ineq}
 &&\int_{\bQ} \left({\mathbb A}_\eta[x]\sigma_\eta(x,t), \sigma_\eta(x,t)\right)dx+
\int_{\bQ} \left(L_\eta[x]z_\eta(x,t), z_\eta(x,t)\right)dx
 \non\\
&&+\int_0^t\int_{\bQ}f_{g_\eta}\left(x, B^T\sigma_\eta(x,\tau)-L_\eta[x]z_\eta(x,\tau), \partial_\tau z_\eta(x,\tau)\right)dxd\tau\\
&& \le\int_{\bQ} \left({\mathbb A}_\eta[x]\sigma^{(0)}_\eta(x), \sigma^{(0)}_\eta(x)\right)dx+
\int_{\bQ} \left(L_\eta[x]z^{(0)}_\eta(x), z^{(0)}_\eta(x)\right)dx+ \left(b, \partial_\tau u_\eta\right)_{{\bQ}_t}\non
\end{eqnarray}
holds for all $t\in(0, T_e)$ and some function $\sigma^{(0)}_\eta\in L^2({\bQ}, {\cal S}^3)$
solving the elliptic boundary value problem \eq{LePr1} - \eq{LePr3}.

The above computations suggest the following notion of weak solutions for 
the initial boundary value problem (\ref{MPr1}) - (\ref{MPr5}).
\begin{Def}\label{WeakSol} Let the numbers $p, q$ satisfy $1 < q\le 2\le p < \infty, 1/p + 1/q = 1.$
 A function $(u_\eta,\sigma_\eta,z_\eta)$ such that
\[(u_\eta,\sigma_\eta)\in W^{1,q}(0,T_e; W^{1,q}_0({\bQ}, {\mathbb R}^3) \times
 L^{q} ({\bQ}, {\cal S}^3)),\] \[ z_\eta\in   
W^{1,q}(0,T_e; L^q({\bQ},{\mathbb R}^N)),\ \ 
\Sigma_\eta:=B^T\sigma_\eta-L_\eta z_\eta\in L^p({\bQ}_{T_e}, {\mathbb R}^N)\]
with
\[(\sigma_\eta,L^{1/2}_\eta z_\eta)\in L^\infty(0,T_e;L^2({\bQ},{\cal S}^3
\times {\mathbb R}^N))\]
is called a {\rm weak solution} of the initial boundary value
problem (\ref{MPr1}) - (\ref{MPr5}), if for every $t\in (0,T_e)$
the function $(u_\eta(t),\sigma_\eta(t))$ is a weak solution of the boundary value problem
(\ref{MPr1}) - (\ref{MPr2}), (\ref{MPr4}) 
for every given $Bz_\eta(t)\in L^q({\bQ},{\cal S}^3)$, the initial condition (\ref{MPr5}) is satisfied
pointwise for almost every $(x,t)$
 and the inequality (\ref{Eqiv_Fitzpatrick_Ineq}) holds for all $t\in(0, T_e)$
and the function $\sigma_\eta^{(0)}\in L^2({\bQ}, {\cal S}^3)$ determined by equations
\eq{LePr1} - \eq{LePr3}.
\end{Def}
Now, we show that the above definition of weak solutions for (\ref{MPr1}) - (\ref{MPr5})
is consistent. Namely, we are going to prove that if a triple of functions $(u_\eta,\sigma_\eta,z_\eta)$
is a weak solution of (\ref{MPr1}) - (\ref{MPr5}) in the sense of Definition~\ref{WeakSol} and possesses additional regularity, 
then this triple of functions is a solution of the initial boundary value problem 
(\ref{MPr1}) - (\ref{MPr5}), i.e. the constitutive inclusion (\ref{MPr3}) is satisfied pointwise
 for a.e. $(x,t)\in {\bQ}_{T_e}$. 
To this end, we assume that the weak solution $(u_\eta,\sigma_\eta,z_\eta)$ has the following regularity
\[(u_\eta,\sigma_\eta)\in W^{1,1}(0,T_e; H^1_0({\bQ}, {\mathbb R}^3) \times
 L^2 ({\bQ}, {\cal S}^3)),\] \[ z_\eta\in   
W^{1,1}(0,T_e; L^2({\bQ},{\mathbb R}^N)).\]
Then, it is immediately seen that the function $\sigma_\eta^{(0)}\in L^2({\bQ}, {\cal S}^3)$ as a unique solution of the problem \eq{LePr1} - \eq{LePr3} satisfies the relation 
$\sigma^{(0)}_\eta(x)=\sigma_\eta(x, 0)$ for a.e. $x\in {\bQ}$ and the following identity
\[
\big({\mathbb A}_\eta\sigma_\eta(t),\sigma_\eta(t)\big)_{\bQ}-
\big({\mathbb A}_\eta\sigma_\eta^{(0)},\sigma_\eta^{(0)}\big)_{\bQ}=
\int_{{\bQ}_t}\frac{\partial}{\partial \tau}
\left({\mathbb A}_\eta\sigma_\eta(x,\tau), \sigma_\eta(x,s)\right)dsdx\]
Moreover, we have that
\[
\big\|L^{1/2}_\eta z_\eta(t)\big\|^2_{\bQ}-
\big\|L^{1/2}_\eta z_\eta^{(0)}\big\|^2_{\bQ}=
\int_0^t\frac{\partial}{\partial \tau}
\|L^{1/2}_\eta z_\eta(\tau)\|^2_{\bQ}d\tau.\]
Then, the inequality \eq{Eqiv_Fitzpatrick_Ineq}
 can be rewritten as follows
\[\int_{{\bQ}_t}\Big(({\mathbb A}_\eta\partial_\tau\sigma_\eta,\sigma_\eta)+
(L_\eta z_\eta,\partial_\tau z_\eta)
+f_{g_\eta}\left(x, B^T\sigma_\eta-L_\eta z_\eta, \partial_\tau z_\eta\right)\Big)d\tau dx\le
(b,\partial_\tau u_\eta)_{{\bQ}_t}.\]
Handling the equations (\ref{MPr1}) - (\ref{MPr2}) as above
we obtain that the last inequality takes the following form
\[\int_{{\bQ}_t}\Big(
(L_\eta z_\eta,\partial_\tau z_\eta)
+f_{g_\eta}\left(x, B^T\sigma_\eta-L_\eta z_\eta, \partial_\tau z_\eta\right)\Big)d\tau dx\le
(B^T\sigma_\eta,\partial_\tau z_\eta)_{{\bQ}_t}.\]
or, equivalently,
\[\int_{{\bQ}_t}f_{g_\eta}\left(x, B^T\sigma_\eta-L_\eta z_\eta, \partial_\tau z_\eta\right)d\tau dx\le
\int_{{\bQ}_t}(B^T\sigma_\eta-L_\eta z_\eta,\partial_\tau z_\eta)dxd\tau.\]
Therefore, by \eq{Fitzpatrick_relation1} and the standard localization argument we get that
\[f_{g_\eta}\left(x, B^T\sigma_\eta(x,t)-L_\eta[x]z_\eta(x,t), \partial_tz_\eta(x,t)\right)\]\[
=\left(B^T\sigma_\eta(x,t)-L_\eta[x]z_\eta(x,t), \partial_tz_\eta(x,t)\right),\]
which holds for a.e. $(x,t)\in {\bQ}_{T_e}$. Now, based on the
equivalence result  \eq{Fitzpatrick_relation2}
we conclude that the inclusion (\ref{MPr3}) is satisfied pointwise
from the assumed temporal regularity of $(u_\eta,\sigma_\eta,z_\eta)$.
The pointwise meaning of (\ref{MPr5}) follows.

\paragraph{Existence result.} First, we define a class of maximal monotone functions we deal with
in this work.
\begin{Def}\label{CoercClass}
Let $S$ be a measurable set in ${\mathbb R}^s$ and $m\in L^1(S,\mathbb{R})$. 
For $\alpha_1, \alpha_2\in{\mathbb R}_{+}$, 
${\cal M}(S,{\mathbb R}^k,\alpha_1, \alpha_2,m)$ is the set of measurable
multi-valued functions $h:S \to{\mathfrak M}({\mathbb R}^k\times{\mathbb R}^k)$
(in the sense of Definition~\ref{def:meas-max-monotone}) such that with the following inequality
\begin{eqnarray}
\label{inequMain} 
(v,v^*)\ge m(x)+\alpha_1|v^*|^q+
 \alpha_2|v|^p
\end{eqnarray}
holds  for a.e. $x\in S$ and every $v^*\in h(x,v)$, 
where $p$ and $q$ satisfy the relations $2\le p<\infty$ and $q=p/(p-1)$.
\end{Def}
The main properties of the class ${\cal M}(S,{\mathbb R}^k,\alpha_1, \alpha_2,m)$ are
collected in the following proposition (see \cite[Corollary 2.15]{Damlamian07}).
\begin{Prop}\label{MainClassMaxMonoProp}
Let ${\cal H}$ be a canonical extension of a function $h:S \to{\mathfrak M}({\mathbb R}^k\times{\mathbb R}^k)$
in the sense of Definition \ref{CanExtension},
which belongs to
    ${\cal M}(S,{\mathbb R}^k,\alpha_1, \alpha_2,m)$. Then ${\cal H}$
  is maximal monotone, surjective and $D({\cal H})=L^p(S,{\mathbb R}^k)$.
     \end{Prop}
Now, we can state the main result of this section.
\begin{Theor}\label{ExResultPositiveSemiDef} 
 Assume that $L_\eta$ is positive semi-definite, ${\mathbb C}_\eta$ is uniformly positive definite
  and ${\mathbb C}_\eta\in C(\bar{\bQ}, {\cal L}({\cal S}^3,{\cal S}^3))$,
the mappings
$g_\eta\in {\cal M}({\bQ},{\mathbb R}^N,\alpha_1, \alpha_2,m)$ with
a function $m$ from $L^1({\bQ}, {\mathbb R})$. Suppose that $b \in W^{1,p}(0,T_e; W^{-1,p}({\bQ}, {\mathbb R}^3))$ and
$z_\eta^{(0)}\in L^2({\bQ}, {\mathbb R}^N)$.

Then the initial boundary value
problem (\ref{MPr1}) - (\ref{MPr5}) has at least one weak solution $(u_\eta,T_\eta,z_\eta)$ in the sense of Definition~\ref{WeakSol}. 
\end{Theor}
\begin{Rem}
We point out that the requirement of the continuity of ${\mathbb C}_\eta$ is superfluous and is only made to simplify the proof of Theorem~\ref{ExResultPositiveSemiDef}. The proof itself works for the measurable function ${\mathbb C}_\eta$ as well. The continuity assumption allows us
 to apply the $L^p$-regularity theory for linear elliptic systems in \cite{Giusti2003} directly to our problem.  In case of ${\mathbb C}_\eta\in L^\infty({\bQ}, {\cal L}({\cal S}^3,{\cal S}^3))$, 
 some extra technical work
 has to be done before one can use the $L^p$-regularity theory for linear elliptic systems 
 (this strategy is realized in \cite{NesenenkoNeff2012}). 
 To avoid the technicalities we assume the continuity of ${\mathbb C}_\eta$ here.
\end{Rem}
\begin{proof} To simplify the notations we drop $\eta$.
The proof of the theorem is presented in \cite{Nesenenko12a}.
Therefore, we only sketch it here. We show this by
 the Rothe method (a time-discretization
method, see \cite{Roubi05} for details).  
In order to introduce a time-discretized problem, let us fix any
$m\in{\mathbb N}$ and set
\[h=h_m:=\frac{T_e}{2^m}, \ z^0_m:=z^{(0)},\ b^n_m:=
\frac{1}{h}\int^{nh}_{(n-1)h}b(s)ds\in 
W^{-1,p}({\bQ}, {\mathbb R}^3),\ \ n=1,...,2^m. \]
We are looking for functions $u^n_m\in H^1_0({\bQ},{\mathbb R}^3)$,
$\sigma^n_m\in L^2({\bQ},{\cal S}^3)$ and $z^n_m\in
 L^2({\bQ},{\mathbb R}^N)$ with 
\[\Sigma_{n,m}:=B^T\sigma^n_m-
\frac{1}{m}z^n_m-L z^n_m\in L^p({\bQ},{\mathbb R}^N)\]
solving the following problem  
\begin{eqnarray}
- \di_x \sigma^n_m(x) &=&  b^n_m(x), \label{CurlPr1Dis}
\\[1ex]
\sigma^n_m(x) &=& {\mathbb C}[x](\varepsilon(\na_xu^n_m(x)) - Bz^n_m(x)), 
  \label{CurlPr2Dis}
\\[1ex] 
\label{CurlPr3Dis} \frac{z^n_m(x)-z^{n-1}_m(x)}{h} & \in & 
g\big(x,\Sigma_{n,m}(x)\big),\label{microPr3Dis} 
\end{eqnarray}
together with the boundary conditions 
\begin{eqnarray} 
u^n_m(x) &=& 0, \quad x \in \partial{\bQ}\,.\label{CurlPr6Dis}
\end{eqnarray}
The proof of the existence of the triple $$(u^n_m,
\sigma^n_m, z^n_m)\in H_0^1({\bQ}, {\mathbb R}^3)\times L^2({\bQ},{\cal S}^3)\times
 L^2({\bQ},{\mathbb R}^N)$$ satisfying
 \eq{CurlPr1Dis} - \eq{CurlPr6Dis} can be found in \cite{Nesenenko12a}.\vspace{1ex}\\
{\bf A-priori estimates.} 
Multiplying (\ref{CurlPr1Dis}) by $(u^n_m-u^{n-1}_m)/h$ and then integrating
over ${\bQ}$ we get
\begin{eqnarray}
\big(\sigma^n_m, \varepsilon(\na_x(u^n_m-u^{n-1}_m))/h\big)_{\bQ}=
\left(b^n_m,(u^n_m-u^{n-1}_m)/h\right)_{\bQ}. \label{CurlPr2Disestimate}
\end{eqnarray}
Applying $g^{-1}(x)$ to both sides of (\ref{microPr3Dis}), 
multiplying by $w^n_m:=(z^n_m-z^{n-1}_m)/h$ and
then integrate over ${\bQ}$ to obtain
\[\int_{\bQ} \big(g^{-1}(w^n_m),w^n_m\big)dx=
(\sigma^{n}_m,Bw^n_m)_{\bQ}
-\frac{1}{mh}\Big(z^n_m-z^{n-1}_m, z^n_m\Big)_{\bQ}
-\frac{1}{h}\Big(z^n_m-z^{n-1}_m, Lz^n_m\Big)_{\bQ}.\]
With (\ref{CurlPr2Disestimate}) we get that
\[\frac{1}{h}\Big({\mathbb C}^{-1}\sigma^n_m,
\sigma^n_m-\sigma^{n-1}_m\Big)_{\bQ}+ \frac{1}{h}
\Big(L^{1/2}(z^n_m-z^{n-1}_m), L^{1/2}z^n_m\Big)_{\bQ}\]
\[
+\frac{1}{m} \frac{1}{h}\Big(z^n_m-z^{n-1}_m, z^n_m\Big)_{\bQ}
+\int_{\bQ} \big({g}^{-1}(w^n_m),w^n_m\big)dx
=\frac{1}{h}\left(b^n_m,u^n_m-u^{n-1}_m\right)_{\bQ}.\]
Multiplying by $h$ and summing the obtained relation for $n=1,...,l$ 
for any fixed $l\in[1,2^m]$ we derive the following inequality (${\mathbb A}={\mathbb C}^{-1}$)
\begin{eqnarray}
&& \frac{1}2\Big(
\|{\mathbb A}^{1/2}\sigma^l_m\|^2_{\bQ}+\|L^{1/2}z^l_m\|^2_{\bQ}
+\frac{1}{m}\|z^l_m\|^2_{\bQ}\Big)
+h\sum^l_{n=1}\int_{\bQ}\big({g}^{-1}(w^n_m),w^n_m\big)dx\non\\
&&\label{AprioriEstimHelp1}\hspace{15ex}
\le
C^{(0)}+
h\sum^l_{n=1}\left(b^n_m,\frac{u^n_m-u^{n-1}_m}h\right)_{\bQ},\label{AprioriEstimHelp11}
\end{eqnarray}
where 
\[2C^{(0)}:=\|{\mathbb A}^{1/2}\sigma^0_m\|^2_{\bQ}+\|L^{1/2} z^0_m\|^2_{\bQ}
+\frac{1}{m}\|z^0_m\|^2_{\bQ}.\]
We estimate now the right hand side of the last inequality. Since $u^n_m$
is a solution of the linear elliptic problem formed by the equations 
(\ref{CurlPr1Dis}), (\ref{CurlPr2Dis}) and
(\ref{CurlPr6Dis}), it satisfies (see \cite{Giusti2003}) the inequality
\begin{eqnarray}\label{AprioriEstimHelp2}
\|u^n_m\|_{1,q,{\bQ}}\le C\big(\|b^n_m\|_{q,{\bQ}}+\|z^n_m\|_{q,{\bQ}}\big),
\end{eqnarray}
where $C$ is a positive constant independent of $n$ and $m$. Therefore,
using the linearity of the problem formed by 
(\ref{CurlPr1Dis}), (\ref{CurlPr2Dis}) and (\ref{CurlPr6Dis}), the inequality
(\ref{AprioriEstimHelp2}) and Young's
inequality with $\epsilon>0$ we get that
\begin{eqnarray}\label{AprioriEstimHelp3} 
&&\left(b^n_m,\frac{u^n_m-u^{n-1}_m}h\right)_{\bQ}\le \|b^n_m\|_{p,{\bQ}}
\|({u^n_m-u^{n-1}_m})/h\|_{1,q,{\bQ}}\le
CC_\epsilon\|b^n_m\|^p_{p,{\bQ}}\non\\
&&+\epsilon C\|(b^n_m-b^{n-1}_m)/h\|^q_{q,{\bQ}}
+\epsilon C\|(z^n_m-z^{n-1}_m)/h\|^q_{q,{\bQ}},
\end{eqnarray}
where $C_\epsilon$ is a positive constant appearing in the Young
inequality.
Combining the inequalities (\ref{AprioriEstimHelp11}) and 
(\ref{AprioriEstimHelp3}), applying \eqref{Fitzpatrick_relation1} and \eq{inequMain}
and choosing an appropriate value for $\epsilon>0$ we obtain 
the following estimate
\begin{eqnarray}
&&\frac{1}2\Big(
\|{\mathbb A}^{1/2}\sigma^l_m\|^2_{\bQ}+\|L^{1/2}z^l_m\|^2_{\bQ}
+\frac{1}{m}\|z^l_m\|^2_{\bQ}\Big)+h\hat{C}_\epsilon\sum^l_{n=1}
\int_{\bQ}\Big|\frac{z^n_m-z^{n-1}_m}h\Big|^qdx\non\\
&&\hspace{10ex}\label{aprioriEstim1N} \le
C^{(0)}+h\tilde{C}_\epsilon\sum^l_{n=1}\Big(\|b^n_m\|^p_{p,{\bQ}}+
\|(b^n_m-b^{n-1}_m)/h\|^q_{q,{\bQ}}\Big),
\end{eqnarray}
where $\tilde{C}, \tilde{C}_\epsilon$ and $\hat{C}_\epsilon$ are some positive
constants. Now, using the definition
of Rothe's approximation functions (see \eqref{RotheConstantinterpolant}) we rewrite (\ref{aprioriEstim1N}) as follows
\begin{eqnarray}
&&
\|{\mathbb A}^{1/2}\bar \sigma_m(t)\|^2_{\bQ}+C_1\|L^{1/2}\bar z_m(t)\|^2_{\bQ}
+\frac{1}{m}\|\bar z_m(t)\|^2_{\bQ}
\label{aprioriEstim2N} \\
&& \hspace{3ex}+
2\hat{C}_\epsilon\int_0^{T_e}\int_{\bQ}\big|\partial_t z_m(x,t)\big|^qdxdt\le
2C^{(0)}+2\tilde{C}_\epsilon\|b\|^p_{W^{1,p}(0,T_e;L^p({\bQ},{\mathbb R}^3))}.\non
\end{eqnarray}
From the estimate (\ref{aprioriEstim2N}) we get then that
\begin{eqnarray}
&&\{z_m\}_m \ {\rm is}\  {\rm uniformly}\  {\rm bounded}\  {\rm in}
 \ W^{1,q}(0,{T_e};L^q({\bQ},{\mathbb R}^N)),\label{aprioriEstim3N}\\[1ex]
&&\{L^{1/2}\bar z_m\}_m \ {\rm is}\  {\rm uniformly}\  {\rm bounded}\  {\rm in}
 \ L^\infty(0,{T_e};L^{2}({\bQ},{\mathbb R}^N)),\label{aprioriEstim3aa}\\[1ex]
&&\{\bar \sigma_m\}_m\ {\rm is}\  {\rm uniformly}\  {\rm bounded}\  {\rm in}
 \ L^\infty(0,{T_e};L^{2}({\bQ},{\cal S}^3)),\label{aprioriEstim3a}\\[1ex]
&&\left\{\frac1{\sqrt{m}} \bar z_m\right\}_m\ {\rm is}\  {\rm uniformly}\  {\rm bounded}\  {\rm in}
 \ L^\infty(0,{T_e};L^{2}({\bQ},{\mathbb R}^N)).\label{aprioriEstim3b}
\end{eqnarray}
In particular, the uniform boundness of the sequences in 
(\ref{aprioriEstim3N}) - (\ref{aprioriEstim3b}) yields 
\begin{eqnarray}
&&\left\{\bar\Sigma_{m}\right\}_m\ {\rm is}\ 
 {\rm uniformly}\  {\rm bounded}\  {\rm in}
 \ L^p(0,{T_e};L^p({\bQ},{\mathbb R}^N)),\label{aprioriEstim4N}\\[1ex]
&&\{u_m\}_m \ {\rm is}\  {\rm uniformly}\  {\rm bounded}\  {\rm in}
 \  W^{1,q}(0,{T_e};W^{1,q}_0({\bQ},{\mathbb R}^3)).\label{aprioriEstim5N}
\end{eqnarray}
Employing (\ref{RotheEstim}), the estimates 
(\ref{aprioriEstim3aa}) - (\ref{aprioriEstim4N}) further imply that the
 sequnces  $\{\sigma_m\}_m$, $\{L^{1/2} z_m\}_m$,
 $\left\{ z_m/\sqrt{m}\right\}_m$ and 
$\left\{\Sigma_{m}\right\}_m$ are also
uniformly bounded in the spaces $L^\infty(0,{T_e};L^{2}({\bQ},{\cal S}^3))$,
$L^\infty(0,{T_e};L^{2}({\bQ},{\mathbb R}^N))$, 
$L^\infty(0,{T_e};L^{2}({\bQ},{\mathbb R}^N))$ and 
$L^p(0,{T_e};L^{p}({\bQ},{\mathbb R}^N))$, respectively.
Moreover, due to (\ref{aprioriEstim3N}) and the following obvious relation
\[z^l_m=z^0_m+h\sum^l_{n=1}\left(\frac{z^n_m-z^{n-1}_m}h\right)\]
we may conclude that $\{\bar z_m\}_m$ is uniformly bounded in 
$L^q(0,{T_e};L^q({\bQ},{\mathbb R}^N))$.\vspace{1ex}\\
In \cite{Nesenenko12a} it is shown that  the limit functions denoted by $u, T, z$ and
$\Sigma$ of the corresponding weakly convergent sequences have the following properties
\[u\in W^{1,q}(0,T_e;W^{1,q}_0({\bQ},{\mathbb R}^3)),\ \ (\sigma,L^{1/2}z) \in 
L^\infty(0,T_e;L^2({\bQ},{\cal S}^3\times{\mathbb R}^N)),\]
and 
\[z\in 
W^{1,q}(0,{T_e};L^q({\bQ},{\mathbb R}^N)),
\ \ \Sigma=B^TT- Lz\in 
L^p({\bQ}_{T_e},{\mathbb R}^N).\]
To prove that the weak limit of $(u_m, T_m,z_m)$ is a weak solution of the problem
 (\ref{MPr1}) - (\ref{MPr5}), we are going to employ the concept of the Fitzpatrick function again.
To this end, we rewrite \eq{AprioriEstimHelp11} as follows
\begin{eqnarray}\label{ConvergenceFitzpatrick1}
&& \frac{1}2\Big(
\|{\mathbb A}^{1/2}\bar\sigma_m(t)\|^2_{\bQ}+\|L^{1/2}\bar z_m(t)\|^2_{\bQ}
+\frac{1}{m}\|\bar{z}_m(t)\|^2_{\bQ}\Big)\\
&&+\int_0^t\int_{\bQ} F_{g}\left(x, B^T\bar\sigma_m-L\bar{z}_m, \partial_\tau z_m\right)dxds\le
C^{(0)}+
\int_0^t\left(\bar{b}_m,\partial_\tau u_m\right)_{\bQ}d\tau.\non
\end{eqnarray}
Next, using the lower semi-continuity of convex functionals we get (\ref{Eqiv_Fitzpatrick_Ineq}) 
after passing to the weak limit in \eq{ConvergenceFitzpatrick1}.
 This completes the proof of Theorem~\ref{ExResultPositiveSemiDef}.
\end{proof}



\section{\label{sec:Preliminaries}Stochastic homogenization}

Throughout this section, we follow the setting for stochastic homogenization proposed in \cite{heida2016a} for rate-independent
systems.
\begin{Rem}\label{rem:palm-lebesgue}
In the following, we introduce the concept of Palm measures. Note that we will need this concept only in the context of the results in Section \ref{sec:hom-convex}. For the main results proved in Section \ref{Homogenization} we will restrict to the case $\muomega^\eta=\lebesgueL$ which implies $\mupalm=\cP$ (this follows from the translation invariance and Fubini's theorem). In this case, we will omit $\mupalm$ and every integral over $\Omega$ is meant with respect to $\cP$. In particular, we will write $\int_\Omega f:=\int_\Omega f(\omega)d\cP(\omega)$.
\end{Rem}
\subsection{\label{sub:Ergodic-dynamical-systems}Concept of Palm measures}
Let $(\Omega,\cF_{\Omega},\cP,\tau)$ be a probability space with
dynamical system satisfying Assumption \ref{assu:Omega-mu-tau} and
let $\cM(\Rn)$ be the set of Radon measures on $\Rn$ equipped with
the Vague topology. 
\begin{Def}
\label{def:random-measure} Let $(\Omega,\cF_{\Omega},\cP,\tau)$
satisfy Assumtion \ref{assu:Omega-mu-tau}. A \emph{random measure}
is a mapping $\mu_{\bullet}:\,\Omega\to\cM(\Rn)$, $\omega\mapsto\mu_{\omega}$
such that $\omega\mapsto\mu_{\omega}(A)$ is measurable for all Borel
sets $A\subset\Rn$. A random measure is called stationary, if $\mu_{\tau_{x}\omega}(A)=\mu_{\omega}(A+x)$
for all Borel sets $A\subset\Rn$. The \emph{intensity} $\lambda(\muomega)$
is defined by: 
\begin{equation}
\lambda(\muomega):=\int_{\Omega}\int_{[0,1]^{n}}d\muomega(x)\,d\cP(\omega)\,.
\end{equation}
\end{Def}
\begin{Theor}[Mecke \cite{Mecke1967,Daley1988}: Existence of Palm measure]
\label{thm:Mecke-Palm}Let $\omega\mapsto\mu_{\omega}$ be a stationary
random measure. Then there exists a unique measure $\mupalm$ on $\Omega$
such that 
\[
\intomega\intrn f(x,\tau_{x}\omega)\,d\muomega(x)d\cP(\omega)=\intrn\intomega f(x,\omega)\,d\mupalm(\omega)dx
\]
 for all $\lebesgueL\times\mupalm$-measurable non negative functions
and all $\lebesgueL\times\mupalm$- integrable functions $f$. Furthermore
for all $A\subset\Omega$, $u\in L^{1}(\Omega,\mupalm)$ there holds
\begin{align}
\mupalm(A) & =\intomega\intrn g(s)\chi_{A}(\tau_{s}\omega)d\muomega(s)d\cP(\omega)\label{eq:def-mupalm}\\
\intomega u(\omega)d\mupalm & =\intomega\intrn g(s)u(\tau_{s}\omega)d\muomega(s)d\cP(\omega)
\end{align}
 for an arbitrary $g\in L^{1}(\Rn,\lebesgueL)$ with $\intrn g(x)dx=1$
and $\mupalm$ is $\sigma$-finite.\end{Theor}
\begin{Rem}
\label{rem:Remark-Palm}a) Setting $g(s):=\chi_{[0,1]^{n}}(s)$, the
Palm measure can equally be defined through (\ref{eq:def-mupalm}).

b) For the constant measure $\omega\mapsto\lebesgueL$, we simply
find $\mupalm=\cP$, the original probability measure. This is a direct
consequence of (\ref{eq:def-mupalm}), Fubini's theorem and Assumption
\ref{assu:Omega-mu-tau} (ii).
\end{Rem}
For a random measure $\muomega$, we define
\begin{equation}
\muomega^{\eta}(A):=\eta^{n}\muomega(\eta^{-1}A)\,.\label{eq:def-mu-omega-eps}
\end{equation}

\begin{Theor}[Ergodic Theorem \cite{Daley1988}]
\label{thm:ergodic-thm-palm}Let Assumption \ref{assu:Omega-mu-tau}
hold for $(\Omega,\cF_{\Omega},\cP,\tau)$. Let $\mu_{\omega}$ be
a stationary random measure with finite intensity and Palm measure
$\mupalm$. Then, for all $g\in L^{1}(\Omega,\mupalm)$ there holds
$\cP$ almost surely 
\begin{equation}
\lim_{\eta\rightarrow0}\int_{A}g(\tau_{\frac{x}{\eta}}\omega)d\muomega^{\eta}(x)=|A|\int_{\Omega}g(\omega)d\mupalm(\omega)\label{eq:palm-ergodic-formula}
\end{equation}
 for all bounded Borel sets $A$.
\end{Theor}
The ergodic theorem only holds for functions on $\Omega$. Nevertheless,
it motivates the following generalization of the concept of ergodicity:
\begin{Def}
\label{def:ergodicity-general-functions}Let $f\in L^{p}(\bQ\times\Omega;\lebesgueL\otimes\mupalm)$
for some $1\leq p<\infty$. 
We say that $f$ is an ergodic function if it has a ${\cal B}(\bQ)\otimes\cF_\Omega$-measurable 
representative $\tilde f$ such that for $a=\tilde f$ and $a=|\tilde f|^p$ it holds
\begin{equation}
\begin{aligned}\lim_{\eta\rightarrow0}\int_{\bQ}a(x,\tau_{\frac{x}{\eta}}\omega)\,d\muomega^{\eta}(x) & =\int_{\bQ}\int_{\Omega}a(x,\tilde{\omega})\,d\mupalm(\tilde{\omega})\,dx\,.\end{aligned}
\label{eq:ergodicity-general-functions}
\end{equation}
Finally, we say $f\in L^{p}(\bQ\times\Omega;\lebesgueL\otimes\mupalm)$
is an ergodic function if (\ref{eq:ergodicity-general-functions})
holds for almost all $\omega$.
\end{Def}
By Theorem \ref{thm:ergodic-thm-palm}, we find that every $g\in L^{1}(\Omega,\mupalm)$
is ergodic. In \cite[Section 2.5]{heida2016a} a larger set of ergodic functions
was identified:
\begin{Lem}
\label{lem:ergodicity-charac-functions} Let Assumption
\ref{assu:Omega-mu-tau} hold for $(\Omega,\cF_{\Omega},\cP,\tau)$.
Let $\bQ\subset\Rn$ be a bounded domain and let $f\in L^{\infty}(\bQ\times\Omega;\lebesgueL\otimes\mupalm)$.
Then, $f$ is an ergodic function. 
\end{Lem}

\subsection{Potentials and Solenoidals}

Let ${\cal BD}(\Rn)$ denote the set of bounded domains in $\Rn$. 
For every $p$ with $1<p<\infty$, we introduce the following
spaces:
\begin{align*}
L_{\rm pot}^{p}(\Rn) & :=\overline{\left\{ g\in L_{loc}^{p}(\Rn;\Rn)\mid\forall\bQ\in{\cal BD}(\Rn)\,\exists f\in W^{1,p}(\bQ):\,\,g=\nabla f\right\} }^{\norm{\cdot}_{L_{loc}^{p}}}\,,\\
L_{{\rm sol}}^{p}(\Rn) & :=\left\{ g\in L_{loc}^{p}(\Rn;\Rn)\mid\forall\bQ\in{\cal BD}(\Rn)\,\exists f\in W_{0}^{1,p}(\bQ)\,:\,\int_{\bQ} g\cdot\nabla f=0\right\} \,.
\end{align*}
On $\Omega$, we introduce the corresponding spaces
\begin{align*}
L_{\rm pot}^{p}(\Omega) & :=\left\{ f\in L^{p}(\Omega;\Rn)\mid f_{\omega}\in L_{\rm pot}^{p}(\Rn)\,\,\cP-\mbox{a.s. and }\int_{\Omega}f=0\right\} \,,\\
L_{{\rm sol}}^{p}(\Omega) & :=\left\{ f\in L^{p}(\Omega;\Rn)\mid f_{\omega}\in L_{{\rm sol}}^{p}(\Rn)\,\,\cP-\mbox{a.s. and }\int_{\Omega}f=0\right\} \,.
\end{align*}
Then, there holds the following orthogonal decomposition.
\begin{Lem}[{\cite[Theorem 3.1.2]{Pankov97}}]
\label{lem:Ortho-Pot-Sol}Let $1<p<\infty$ and $p^{-1}+q^{-1}=1$
and let Assumption \ref{assu:Omega-mu-tau} hold for $(\Omega,\cF_{\Omega},\cP,\tau)$.
Then the following relations hold in the sense of duality between
the spaces $L^{p}(\Omega,\cP)$ and $L^{q}(\Omega,\cP)$:
\[
\left(L_{\rm pot}^{p}(\Omega)\right)^{\bot}=L_{{\rm sol}}^{q}(\Omega)\oplus\Rn\,,\qquad\left(L_{{\rm sol}}^{p}(\Omega)\right)^{\bot}=L_{\rm pot}^{q}(\Omega)\oplus\Rn\,.
\]

\end{Lem}
Every $L_{\rm pot}^{p}(\Omega)$ function can be optained as the ergodic limit of a sequence of gradients with vanishing potentials. The following result can be proved like in 
\cite[Section 2.3]{heida2014stochastic}.
\begin{Lem}\label{lem:vanishing-ergodic-potential}
Let $\v\in L_{\rm pot}^{p}(\Omega)$, $1<p<\infty$. For almost every $\omega$ there exists $C>0$ such that the following holds: For every $\eta>0$ there exists a unique $V_\eta^\omega$ with $\nabla V_\eta^\omega(x)=\v(\tau_{\frac x\eta}\omega)$ and $\|V_\eta\|_{H^1(\bQ)}\leq C\|\v\|_{L_{\rm pot}^{p}(\Omega)}$ for all $\eta>0$. Furthermore,
$$\lim_{\eta\to0}\|V^\omega_\eta\|_{L^p(\bQ)}=0\,.$$
\end{Lem}
Furthermore, we find the following important Korn inequality, which
can be proved like in \cite[Section 2.3]{heida2014stochastic}. 
\begin{Lem}
\label{lem:Korn-Omega}Let Assumption \ref{assu:Omega-mu-tau} hold
for $(\Omega,\cF_{\Omega},\cP,\tau)$. For every $1<p<\infty$ there
exists $C_{p}>0$ such that for all $\v\in L_{\rm pot}^{p}(\Omega;\Rn)$
\begin{equation}
\norm{\v}_{L^{p}(\Omega;\Rnn)}\leq C_{p}\norm{\v^{s}}_{L^{p}(\Omega;\Rnn)}\,.\label{eq:Korn-Omega}
\end{equation}
\end{Lem}

\subsection{\label{sec:Two-scale-convergence}Two-scale convergence: time independent case}

Let Assumption \ref{assu:Omega-mu-tau} hold for $(\Omega,\cF_{\Omega},\cP,\tau)$
and let $\omega\mapsto\mu_{\omega}$ be a stationary random measure
with $\mu_{\omega}^{\eta}$ and $\mupalm$ defined through (\ref{eq:def-mu-omega-eps})
and (\ref{eq:def-mupalm}). For the case $\mu_\omega=\lebesgueL$ we recall Remark \ref{rem:palm-lebesgue}.
\begin{Rem}
\label{rem:es-count-dense-set}The product $\sigma$-algebra $\cB_{Q}\otimes\cF_{\Omega}$
is countably generated and therefore, the space $L^{p}(\bQ\times\Omega)$
is separable (see \cite[Section VI.15, p. 92]{Doob1994}) for every $1\leq p<\infty$.
In particular, for every $1\leq p<\infty$, there exists a countable
dense subset $\Phi_{p}$ of simple functions in $L^{p}(\bQ\times\Omega;\lebesgueL\otimes\mupalm)$. By
Lemma \ref{lem:ergodicity-charac-functions}, every $\phi\in\Phi_{p}$
is an ergodic function and there exists a set $\Omega_{\Phi_{p}}\subset\Omega$
with $\cP(\Omega_{\Phi_{p}})=1$ such that all $\phi\in\Phi_{p}$
satisfy (\ref{eq:ergodicity-general-functions}) (i.e. they admit
ergodic realizations) for all $\omega\in\Omega_{\Phi_{p}}$. This
corresponds to the setting of \cite{heida2016a}.\end{Rem}
\begin{Def}
\label{def:two-scale-conv}Let $1<p,q<\infty$ with $1/p+1/q=1$.
Let $\Phi_{p}$ be the set of Remark \ref{rem:es-count-dense-set}
and let $\omega\in\Omega_{\Phi_{p}}$. Let $\ueta\in L^{q}(\bQ;\muomega^{\eta})$
for all $\eta>0$. We say that $\ueta$ converges (weakly) in two
scales to $u\in L^{q}(\bQ;L^{q}(\Omega,\mupalm))$ and write $\ueta\stackrel{2s}{\weakto}u$
if 
$\sup_\eta\|\ueta\|_{L^{q}(\bQ;\muomega^{\eta})}<\infty$ and if
for all $\phi\in\Phi_{p}$ there holds with $\phi_{\omega,\eta}(x):=\phi(x,\tau_{\frac{x}{\eta}}\omega)$
\[
\lim_{\eta\to0}\int_{\bQ}\ueta\phi_{\omega,\eta}d\muomega^{\eta}=\int_{\bQ}\int_{\Omega}u\phi\,d\mupalm\,d\lebesgueL\,.
\]
Furthermore, we say that $\ueta$ converges strongly in two scales to
$u$, written $\ueta\stackrel{2s}{\to}u$, if for all weakly two-scale
converging sequences $v^{\eta}\in L^{p}(\bQ;\muomega^{\eta})$ with $v^{\eta}\stackrel{2s}{\weakto}v\in L^{p}(\bQ;L^{p}(\Omega,\mupalm))$
as $\eta\to0$ there holds 
\[
\lim_{\eta\to0}\int_{\bQ}\ueta v^{\eta} d\muomega^{\eta}=\int_{\bQ}\int_{\Omega}uv\,d\mupalm\,d\lebesgueL\,.
\]
\end{Def}
\begin{Lem}[Existence of two-scale limits \cite{heida2016a}]
\label{lem:two-scale-limit}Let $\omega\in\Omega_{\Phi_{p}}$, $1<p<\infty$
and $\ueta\in L^{p}(\bQ;\muomega^{\eta})$ be a sequence of functions
such that $\norm{\ueta}_{L^{p}(\bQ;\muomega^{\eta})}\leq C$ for some
$C>0$ independent of $\eta$. Then there exists a subsequence of
$u^{\eta'}$ and $u\in L^{p}(\bQ;L^{p}(\Omega,\mupalm))$
such that $u^{\eta'}\stackrel{2s}{\weakto}u$ and 
\begin{equation}
\norm u_{L^{p}(\bQ;L^{p}(\Omega,\mupalm))}\leq\liminf_{\eta\to0}\left\|u^{\eta'}\right\|_{L^{p}(\bQ;\muomega^{\eta})}\,.\label{eq:two-scale-limit-estimate}
\end{equation}
\end{Lem}
Closely connected with the definition of two-scale convergence and Lemma \ref{lem:two-scale-limit} is the following result.
\begin{Lem}\label{lem:valid-ts-test-function}
Let $f\in L^1(\Omega,\mupalm)$. For almost all $\omega\in\Omega$ it holds: 
for all $\psi\in C_0(\bQ)$ and $\ueta\in L^{q}(\bQ;\muomega^{\eta})$, $\eta>0$,
and $u\in L^{q}(\bQ;L^{q}(\Omega,\mupalm))$ with $\ueta\stackrel{2s}{\weakto}u$ it holds
\[
\lim_{\eta\to0}\int_{\bQ}\ueta\psi(x)f\left(\tau_{\frac x\eta}\omega\right)d\muomega^{\eta}=\int_{\bQ}\int_{\Omega}u\psi f\,d\mupalm\,d\lebesgueL\,.
\]
\end{Lem}
The following Lemma is well known in the periodic case (\cite{allaire1992homogenization})
but also in the stochastic setting (\cite{heida2016a,Zhikov_Pyatnitskii_2006} for
$p=2$). The following version can be proofed along the same lines
as Lemma 6.2 in \cite{heida2016a}.
\begin{Lem}
\label{lem:sto-conver-grad}Let $1<p<\infty$. If $u^{\eta}\in W_{0}^{1,p}(\bQ;\Rn)$
for all $\eta$ with $\|\nabla u^{\eta}\|_{L^{p}(\bQ)}<C$ for $C$
independent from $\eta>0$ then there exists a subsequence $u^{\eta'}$
and functions $u\in W_{0}^{1,p}(\bQ;\Rn)$ and $v\in L^{p}(\bQ;L_{pot}^{p}(\Omega;\Rn))$
such that 
\[
u^{\eta'}\stackrel{2s}{\weakto}u\quad\mbox{and }\quad\nabla u^{\eta'}\stackrel{2s}{\weakto}\nabla u+v\qquad\mbox{as }\eps\to0\,.
\]

\end{Lem}
We finally collect some usefull results.
\begin{Lem}
\label{lem:Every-v-is-a-fB-limit}Let $u\in L^{p}(\bQ;L^{p}(\Omega;\mupalm))$.
Then, for almost every $\omega\in\Omega$, there exists a sequence
$\ueta\in L^{p}(\bQ;\muomega^\eta)$ such that $\ueta\stackrel{2s}{\weakto}u$
as $\eta\to0$.
\end{Lem}

\begin{Lem}
\label{lem:weak-convergence-equivalent-norms}Let $N\in\N$ and let
$A\in L^{\infty}(\bQ;L^{\infty}(\Omega;\sL(\R^{N},\R^{N})))$ be symmetric
and assume $A$ is $\cB_{Q}\otimes\cF_{\Omega}$ -measurable. We furthermore
assume the existence of a constant $\alpha>0$ such that
\begin{equation}
\alpha\left|\xi\right|^{2}\leq\xi A(x,\omega)\xi\leq\frac{1}{\alpha}\left|\xi\right|^{2}\qquad\forall\xi\in\Rn\mbox{ and for }\lebesgueL\times\mupalm\mbox{-a.e. }(x,\omega)\in\bQ\times\Omega\,.\label{eq:Lemma-norm-weak-conv}
\end{equation}
Then, for almost all $\omega\in\Omega_{\Phi_{p}}$ there holds: For
all sequences $\ueta\in L^{2}(\bQ;\muomega^\eta;\R^{N})$ with weak two-scale
limit $u\in L^{2}(\bQ; L^{2}(\Omega;\mupalm;\R^{N}))$ there holds
with $A_{\eta,\omega}(x):=A(x,\tau_{\frac{x}{\eta}}\omega)$
\[
\liminf_{\eta\to0}\int_{\bQ}\ueta\cdot(A_{\eta,\omega}\ue)\,d\muomega^\eta\geq\int_{\bQ}\int_{\Omega}u\cdot(Au)\,d\mupalm\,.
\]
\end{Lem}

\subsection{Two-scale convergence: time dependent case}

We are also interested in the convergence behavior of functions $\ueta:\,[0,T]\to L^{p}(\bQ,\muomega^\eta)$.
\begin{Def}
\label{def:weak-A-conv-time}Let $1<r,r',p,q<\infty$ with $\frac{1}{p}+\frac{1}{q}=1$
and $\frac{1}{r'}+\frac{1}{r}=1$. Let $\Phi_{q}$ be the set of Remark
\ref{rem:es-count-dense-set} and let $\omega\in\Omega_{\Phi_{q}}$.
Let $\ueta\in L^{r}(0,T;L^{p}(\bQ;\muomega^{\eta}))$ for all $\eta>0$.
We say that $\ueta$ converges (weakly) in two scales to $u\in L^{r}(0,T;L^{p}(\bQ;L^{p}(\Omega,\mupalm)))$
and write $\ueta\stackrel{2s}{\weakto}u$ if for all continuous and
piecewise affine functions $\phi:\,[0,T]\to\Phi_{q}$ there holds
with $\phi_{\omega,\eta}(t,x):=\phi(t,x,\tau_{\frac{x}{\eta}}\omega)$
\[
\lim_{\eta\to0}\int_{0}^{T}\int_{\bQ}\ueta\phi_{\omega,\eta}d\muomega^{\eta}dt=\int_{0}^{T}\int_{\bQ}\int_{\Omega}u\phi\,d\mupalm\,dx\,dt
\]
\end{Def}
The following two lemmas where proved in \cite{heida2016a}.
\begin{Lem}
\label{lem:weak-fA-conv-time}Asssume that $1<p<\infty$ and $1<r\leq\infty$. Then,
every sequence of functions $\left(\ueta\in L^{r}(0,T;L^{p}(\bQ;\muomega^{\eta}))\right)_{\eps>0}$
satisfying $$\norm{\ueta}_{L^{r}(0,T;L^{p}(\bQ;\muomega^{\eta}))}\leq C$$
for some $C>0$ independent from $\eta$ has a weakly two-scale convergent
subsequence with limit function $u\in L^{r}(0,T;L^{p}(\bQ;L^{p}(\Omega,\mupalm)))$.
Furthermore, if $$\norm{\partial_{t}\ueta}_{L^{r}(0,T;L^{p}(\bQ;\muomega^{\eta}))}\leq C$$
uniformly for $1<p\leq\infty$, then also $\norm{\partial_{t}u}_{L^{r}(0,T;L^{p}(\bQ;L^{p}(\Omega,\mupalm)))}\leq C$
and $\partial_{t}\ueta\stackrel{2s}{\weakto}\partial_{t}u$ in the sense
of Definition \ref{def:weak-A-conv-time} as well as $\ueta(t)\stackrel{2s}{\weakto}u(t)$
for all $t\in[0,T]$.
\end{Lem}


\section{Homogenization of convex functionals}\label{sec:hom-convex}
\begin{Lem}
\label{lem:General-Hom-Convex}Let Assumption \ref{assu:Omega-mu-tau}
hold for $(\Omega,\cF_{\Omega},\cP,\tau)$ and the random measure
$\muomega$. Let $f:\,\bQ\times\Omega\times\R^{N}\to\R$ be a convex
functional in $\R^N$. For almost all $\omega\in\Omega_{\Phi_{p}}$
the following holds: Let $\ueta\in L^{q}(\bQ;\R^N)$ be a sequence such that
$\norm{\ueta}_{L^{q}(\bQ)}\leq C$ for some $0<C<\infty$ and such that
$\ueta\stackrel{{\scriptstyle 2s}}{\weakto}u\in L^{q}(\bQ\times\Omega;\lebesgueL\otimes\mupalm;\R^N)$.
Then, it holds 
\[
\int_{\bQ}\int_{\Omega}f(x,\tilde{\omega},u(x,\tilde{\omega}))\,d\mupalm(\tilde{\omega})\,dx\leq\liminf_{\eta\to0}\int_{\bQ}f(x,\tau_{\frac{x}{\eta}}\omega,\ueta(x))\,d\muomega^\eta(x)\,.
\]
\end{Lem}
The proof of Lemma \ref{lem:General-Hom-Convex} is literally the
same as for Theorem 7.1 in \cite{zhikov2004two}. However, we provide
it here for completeness.
\begin{proof}
Let $\omega\in\Omega_{\Phi_{p}}$ and let $f^{\ast}$ denote the Fenchel
conjugate of $f$ in the third variable. Without loss of generality,
we may assume that 
\begin{equation}
\lim_{\eta\to 0}\int_{\bQ}f_{\eta,\omega}^{\ast}(x,\psi_{\omega}^{\eta}(x))\,d\muomega^\eta(x)=\int_{\bQ}\int_{\Omega}f^{\ast}(x,\tilde{\omega},\psi(x,\tilde{\omega}))\,d\mupalm(\tilde{\omega})\,dx\label{eq:Lem-GHC-help-1}
\end{equation}
for all $\psi\in\Phi_{p}^N$ and all $\omega\in\Omega_{\Phi_{p}}$.
We first consider the case 
\begin{equation}
f(x,\tilde{\omega},\xi)\geq\left|\xi\right|^{q}\label{eq:Lem-GHC-help-2}
\end{equation}
for almost every $(x,\tilde{\omega})\in\bQ\times\Omega$ and all $\xi\in\R^{N}$.
We then find for every $\psi\in\Phi_{p}^N$ 
\[
F_{\eta}:=\int_{\bQ}f_{\eta,\omega}(x,\ueta(x))\,d\muomega^\eta(x)\geq\int_{\bQ}\ueta(x)\cdot\psi_{\omega}^{\eta}(x)d\muomega^\eta(x)-\int_{\bQ}f_{\eta,\omega}^{\ast}(x,\psi_{\omega}^{\eta}(x))\,d\muomega^\eta(x)\,.
\]
Due to $\ueta\tsweak u$ and (\ref{eq:Lem-GHC-help-1}) we find 
\[
\liminf_{\eta\to0}F_{\eta}\geq\int_{\bQ}\int_{\Omega}\left(u(x,\tilde{\omega})\cdot\psi(x,\tilde{\omega})-f^{\ast}(x,\tilde{\omega},\psi(x,\tilde{\omega}))\right)\,d\mupalm(\tilde{\omega})\,dx
\]
for all $\psi\in\Phi_{p}^N$. Since (\ref{eq:Lem-GHC-help-2}) holds,
$f^{\ast}$ is continuous in $\xi$ and the last inequality implies
\begin{equation}
\liminf_{\eta\to0}F_{\eta}\geq\int_{\bQ}\int_{\Omega}f(x,\tilde{\omega},u(x,\tilde{\omega}))\,d\mupalm(\tilde{\omega})\,dx\,.\label{eq:lem:General-Hom-Convex-prelim}
\end{equation}
In the general case, let 
\[
F_{\eta}^{\delta}:=\int_{\bQ}f_{\eta,\omega}(x,\ueta(x))\,d\muomega^\eta(x)+\delta\norm{\ueta}_{L^{q}(\bQ)}.
\]
Then, $0<F_{\eta}^{\delta}-F_{\eta}\leq C\delta$ and (\ref{eq:lem:General-Hom-Convex-prelim})
implies that 
\[
\liminf_{\eta\to0}F_{\eta}^{\delta}\geq\int_{\bQ}\int_{\Omega}f(x,\tilde{\omega},u(x,\tilde{\omega}))\,d\mupalm(\tilde{\omega})\,dx+\delta\norm u_{L^{q}(\bQ\times\Omega)}\,.
\]
Hence the claim follows.\end{proof}
\begin{Lem}
Let Assumption \ref{assu:Omega-mu-tau} hold for $(\Omega,\cF_{\Omega},\cP,\tau)$
and let $\muomega$ be a random measure. Let $f:\,\bQ\times\Omega\times\R^{N}\to\R$
be such that for a.e. $(x,\tilde{\omega})$ the function $f(x,\tilde{\omega},\cdot)$
is convex in $\R^{N}$. Then, for almost every $\omega\in\Omega_{\Phi_{p}}$
the following holds: If $\ueta_{\omega}\in L^{q}(\bQ;\R^N)$ is a sequence
of minimizers of the functionals $$F_{\eta,\omega}:\,u\mapsto\int_{\bQ}f(x,\tau_{\frac{x}{\eta}}\omega,u(x))\,d\muomega^\eta(x)$$
and if $\sup_{\eta>0}\norm{u_{\omega}^{\eta}}_{L^{q}(\bQ;\muomega^\eta)}<\infty$,
then there exists $u_{0,\omega}\in L^{q}(\bQ\times\Omega;\lebesgueL\otimes\mupalm;\R^N)$
such that $\ueta_{\omega}\tsweak u_{0,\omega}$ along a subsequence
and $u_{0,\omega}$ is a minimizer of 
\begin{align*}
F_{0}\,:\,\,\,L^q(\bQ\times\Omega;\lebesgueL\otimes\mupalm;\R^N) & \to\R\\
u & \mapsto\int_{\bQ}f(x,\tilde{\omega},u(x,\tilde{\omega}))\,d\mupalm(\tilde{\omega})\,dx\,.
\end{align*}
\end{Lem}
\begin{proof}
Let $u_{0}\in L^{q}(\bQ\times\Omega;\lebesgueL\otimes\mupalm)$ be
a minimizer of $F_{0}$. By \cite[Theorem III-39]{castaingvaladier1977convex}
we can assume that $u_{0}(x,\tilde{\omega})$ minimizes $f(x,\tilde{\omega})$
for almost every $(x,\tilde{\omega})$. Then, for almost all $\omega\in\Omega$
it holds $\ueta_{0}(x):=u_{0}(x,\tau_{\frac{x}{\eta}}\omega)\in L^{q}(\bQ;\R^N)$
and 
\[
F_{\eta,\omega}(\ueta_{0})\geq F_{\eta,\omega}(\ueta)\,.
\]
We chose a subsequence $u_{\omega}^{\eta'}$ and $u_{0,\omega}\in L^{q}(\bQ\times\Omega;\lebesgueL\otimes\mupalm;\R^N)$
such that $u_{\omega}^{\eta'}\tsweak u_{0,\omega}$. Since 
$F_{\eta,\omega}(\ueta_{0})\to F_{0}(u_{0})$, we find 
\[
\int_{\bQ}\int_{\Omega}f(x,\tilde{\omega},u_{0,\omega}(x,\tilde{\omega}))\,d\mupalm(\tilde{\omega})\,dx\leq\liminf_{\eta\to0}F_{\eta,\omega}(\ueta_{\omega})\leq F_{0}(u_{0}).
\]
Hence, $u_{0,\omega}$ is a minimizer of $F_{0}$. \end{proof}
\begin{Theor}
Let Assumption \ref{assu:Omega-mu-tau} hold for $(\Omega,\cF_{\Omega},\cP,\tau)$
and let $\muomega$ be a random measure and let $1<p,q<\infty$ with
$\frac{1}{p}+\frac{1}{q}=1$. Let $\bQ\subset\Rn$ be a bounded domain and $f:\,\bQ\times\Omega\times\Rd\to\R$
be measurable and for all $(x,\omega)$. Let $f(x,\omega,\cdot)$ be
continuous and convex in $\Rd$ with $f(x,\omega,\xi)\geq|\xi|^{q}$.
Then, for almost every $\omega\in\Omega_{\Phi_{p}}$ it holds: If
$\ueta\in W^{1,q}(\bQ)$ is a sequence of minimizers 
of the functional 
\begin{align*}
F_{\eta}:\,\,L^{q}(\bQ) & \to\R\\
u & \mapsto\int_{\bQ}f_{\omega,\eta}(x,\nabla u(x))\,dx
\end{align*}
such that $\sup_{\eta>0}\norm{u_{\omega}^{\eta}}_{L^{q}(\bQ;\muomega^\eta)}<\infty$,
then there exist $u_{\omega}\in W_{0}^{1,q}(\bQ)$ and 
$\v_{\omega}\in L^{q}(\bQ; L_{\rm pot}^{q}(\Omega))$
and a subsequence $u^{\eta'}$ such that $u^{\eta'}\to u_{\omega}$
strongly in $L^{q}(\bQ)$ and $\nabla u^{\eta'}\tsweak\nabla u_{\omega}+\v_{\omega}$
as $\eta\to0$ and $(u_{\omega},\v_{\omega})$ is a minimizer of the
functional 
\begin{align*}
F_0:\,\,W_{0}^{1,q}(\bQ)\times L^{q}(\bQ; L_{\rm pot}^{q}(\Omega)) & \to\R\\
(u,\v) & \mapsto\int_{\bQ}\int_{\Omega}f(x,\omega,\nabla u(x)+\v(x,\tilde{\omega}))\,d\cP(\tilde{\omega})\,dx\,.
\end{align*}
\end{Theor}
\begin{proof}
Let  $\Phi_{\rm pot}$ be a countable dense subset of $L^q_{\rm pot}(\Omega)$ 
and let $\tilde\Omega\subset\Omega_{\Phi_q}$ be a set of full measure such that 
Lemma {lem:valid-ts-test-function} holds for all $\omega\in\tilde\Omega$. 
By ${\rm span} \Phi_{\rm pot}$ we denote finite linear combinations of elements of $\Phi_{\rm pot}$.
In what follows we restrict to the case $\omega\in\tilde\Omega$. 

Due to Lemma \ref{lem:sto-conver-grad} there exist $u_{\omega}\in W^{1,q}(\bQ)$
and $\v_{\omega}\in L^{q}(\bQ; L_{\rm pot}^{q}(\Omega))$ such that $\nabla\ueta\tsweak\nabla u_{\omega}+\v_{\omega}$
and $u^{\eta}\to u_{\omega}$ along a subsequence, which we denote
$u^{\eta}$ for simplicity. Let $u_{0}\in W_{0}^{1,q}(\bQ)$ and $\v_{0}\in L^{q}(\bQ; L_{\rm pot}^{q}(\Omega))$
be a minimizer of the functional $F_{0}$. 

Now, let $\delta>0$. There exists $\v_\delta\in L^{q}(\bQ; \R \Phi_{\rm pot})$ 
which is simple and has compact support in $\bQ$ such that 
$\|\v_0-\v_\delta\|_{L^{q}(\bQ; L_{\rm pot}^{q}(\Omega))}<\delta$. 
In particular, we find sets $A_i\subset{\bQ}$, $1\leq i\leq K_\delta$ and 
functions $\hat\v_i\in {\rm span} \Phi_{\rm pot}$ such that
$$ \v_\delta(x,\omega)=\sum_{i=1}^{K_\delta}\chi_{A_i}(x)\hat\v_i(\omega)\,.$$
Let $\left(\varphi_\eps\right)_{\eps>0}\subset C^\infty_0({\rm B}_\eps)$ be a family of mollifiers. For 
$\eps>0$ we denote $\v_{\eps,\delta}(\cdot,\omega):=\varphi_\eps\ast \v(\cdot,\omega)$, 
where $\ast$ is the convolution with respect to the $\bQ$-coordinate. 
Then $\v_{\eps,\delta}\in C^1_0(\overline\bQ; {\rm span} \Phi_{\rm pot})$ for $\eps>0$ small enough and 
$\|\v_{\eps,\delta}-\v_\delta\|_{L^{q}(\bQ; L_{\rm pot}^{q}(\Omega))}\to0$ as $\eps\to0$. 

Given $x\in\bQ$ 
we apply Lemma \ref{lem:vanishing-ergodic-potential} and denote 
$V_{\eta,\eps,\delta}^\omega(x,\cdot)\in H^1(\bQ)$ the $\eta$-potential to 
$\v_{\eps,\delta}(x)$ and $\eta$ and $V_{\eta,\delta}^\omega(x,\cdot)\in H^1(\bQ)$ 
the potential to $\v_{\delta}(x)$ and $\eta$. 
Further, if $\hat V_{i,\eta}^\omega$ is the corresponding $\eta$-potential to $\hat\v_i$, we find 
$$ V_{\eta,\eps,\delta}^\omega(x,z)=\sum_{i=1}^{K_\delta}(\chi_{A_i}\ast\varphi_\eps)(x)\hat V_{i,\eta}^\omega(z)\quad{\rm and}\quad
V_{\eta,\delta}^\omega(x,z)=\sum_{i=1}^{K_\delta}\chi_{A_i}(x)\hat V_{i,\eta}^\omega(z)$$
Since the mapping $\hat\v_i\mapsto \hat V_{i,\eta}^\omega$ is linear, 
we find $V_{\eta,\eps,\delta}^\omega\in C^1_0(\bQ;H^1(\bQ))$ with
\begin{align*}
\nabla \left(V_{\eta,\eps,\delta}^\omega(x,x)\right)& =\nabla_xV_{\eta,\eps,\delta}^\omega(x,x)+\nabla_z V_{\eta,\eps,\delta}^\omega(x,x)\\
&=\left(V_{\eta,\delta}^\omega(\cdot,x)\ast \nabla\varphi_\eps\right)(x)
	+\left(\varphi_\eps\ast \v_{\delta}(\cdot,\tau_{\frac x\eta})\right)(x)
\end{align*}
For the first term on the right hand side we obtain 
\begin{align*}
\int_{\bQ}\left|\left(V_{\eta,\delta}^\omega(\cdot,x)\ast \nabla\varphi_\eps\right)(x)\right|^q\d x
& \leq \int_{\Rd}\d x\int_{\Rd}\d z \left\|\nabla\varphi_\eps\right\|^q_{\infty}\left|V_{\eta,\delta}^\omega(z-x,x)\right|^q\\
& \leq \left\|\nabla\varphi_\eps\right\|^q_{\infty}\int_{\bQ}\d x\sum_{i=1}^{K_\delta}\left|\hat V_{i,\eta}^\omega(x)\right|^q 
\end{align*}

Since the last expression on the right hand side converges to $0$ as $\eta\to0$ by Lemma 
\ref{lem:vanishing-ergodic-potential}, we find that $\nabla V_{\eta,\eps,\delta}(x,x)\tsweak \v_{\eps,\delta}(x,\omega)$.
 
Hence, we find for $\eps$ small enough that $V_{\eta,\eps,\delta}^\omega(x,x)$ is a valid point of evalutation for $F_{\eta}$ and
\begin{align*}
F_{\eta}(u_{0}+V_{\eta,\eps,\delta}^\omega) & =
\int_{\bQ}f_{\omega,\eta}(x,\nabla u_{0}(x)+\left(V_{\eta,\delta}^\omega(\cdot,x)\ast \nabla\varphi_\eps\right)(x)+\left(\varphi_\eps\ast \v_{\delta}(\cdot,\tau_{\frac x\eta}\omega)\right)(x))\,dx\\
 & \to\int_{\bQ}\int_{\Omega}f(x,\omega,\nabla u_{0}(x)+\v_{\eps,\delta}(x,\tilde{\omega}))\,d\cP(\tilde{\omega})\,dx\qquad\text{as }\eta\to0\,.
\end{align*}
On the other hand, due to Lemma \ref{lem:General-Hom-Convex}, we
have 
\[
\lim_{\eta\to0}F_{\eta}(u_{0}+V_{\eta,\eps,\delta}^\omega)\geq\liminf_{\eta\to0}F_\eta(\ueta)\geq F_0(u_{\omega},\v_{\omega})\geq F_0(u_{0},\v_{0})\,.
\]
Since $f$ is continuous in $\xi$, we obtain from successively
passing to the limits $\eps\to0$ and $\delta\to0$
that 
\[
F_0(u_{0},\v_{0})\geq F_0(u_{\omega},\v_{\omega})\geq F_0(u_{0},\v_{0})\,.
\]
\end{proof}



\section{Homogenized system of equations}\label{Homogenization}

In this section, we are in the setting $\mu_\omega=\lebesgueL$ for all $\omega$. Hence, we frequently use the notations introduced in Remark \ref{rem:palm-lebesgue}.

The model equations of the problem (the microscopic problem) are
\begin{eqnarray}
\label{microPr1}  - \di_x \sigma_\eta(x,t) &=&  b(x,t),  \\
\label{microPr2} \sigma_\eta(x,t) &=&  \tilde{\mathbb C}\left[\tau_{\frac{x}{\eta}}\tilde\omega\right](\varepsilon
(\na_xu_\eta(x,t))-B z_\eta(x,t)),\\
\label{microPr3} \partial_t z_\eta(x,t) & \in &
\tilde g\left(\tau_{\frac{x}{\eta}}\tilde\omega, B^T \sigma_\eta(x,t) - \tilde{L}_\eta[\tau_{\frac{x}{\eta}}\tilde\omega]z_\eta(x,t)\right),
\end{eqnarray}
together with the homogeneous Dirichlet boundary condition
\begin{eqnarray}
\label{microPr4}  u_\eta(x,t)=0, \hspace{5ex} (x,t)\in \partial\bQ\times(0,\infty),
\end{eqnarray} 
and the initial condition
\begin{eqnarray}
\label{microPr5}  z_\eta(x,0)=\tilde{z}^{(0)}(x, \tau_{\frac{x}{\eta}}\tilde\omega),
 \hspace{12ex} x\in\bQ.
\end{eqnarray} 
Now, we state the main result on the stochastic homogenization of the weak solution $(u_\eta, \sigma_\eta, z_\eta)$ of 
problem (\ref{microPr1}) - (\ref{microPr5}).
\begin{Theor}\label{HomogEquations}
Suppose that all assumptions of Theorem~\ref{ExResultPositiveSemiDef} are satisfied.
 Let $(u_\eta, \sigma_\eta, z_\eta)$ be a weak solution of the initial-boundary value
problem (\ref{microPr1}) - (\ref{microPr5}). 
Then, there exist 
\begin{eqnarray}
u_0\in  W^{1,q}(0,T_e;W^{1,q}_0(\bQ,{\mathbb R}^3)),\ \ 
u_1\in  W^{1,q}(0,T_e; L^q(\bQ; L^q_{\rm pot}(\Omega; {\mathbb R}^3))),\non\\[1ex]
\sigma_0\in  H^{1}(0,T_e; L^2(\bQ; L^2_{\rm sol}(\Omega; {\cal S}^3))),\ \
z_0\in W^{1,q}(0,T_e; L^q(\bQ; L^q(\Omega; {\mathbb R}^N)))\non
\end{eqnarray}
 such that (up to a subsequence)
\begin{eqnarray}
&&u_\eta\stackrel{2s}{\weakto} u_0, \ \ \na u_\eta\stackrel{2s}{\weakto} \na{u_0}+u_1,\ \
\sigma_\eta\stackrel{2s}{\weakto} \sigma_0\ and \  z_\eta\stackrel{2s}{\weakto} z_0.\label{WeakTwoScaleLimits}
\end{eqnarray}
The weak two-scale limit function $(u_0,u_1,\sigma_0, z_0)$ solves the following homogenized
 problem:
\begin{eqnarray}
\label{HomogEqua1}  -\di_x \left(\int _\Omega\sigma_0(x,\omega,t) d\cP\right)  &=& b(x, t), \\
\non  \sigma_0(x, \omega, t) &=& \tilde{\mathbb C}[\omega] \Big( \ve( u_1(x, \omega, t)) 
  - Bz_0(x, \omega, t)  
\\ 
&&\hspace{17ex} \mbox{} + \ve (\na_x u_0(x, t)) \Big), \label{HomogEqua3}   
\end{eqnarray}       
which hold for $(x, \omega, t) \in \bQ\times\Omega\times [0,T_e ]$, and with
 the boundary condition  
\begin{equation}\label{HomogEqua6}  
u_0(x, t) = 0,\quad (x, t) \in \partial {\bQ} \times (0, T_e).  
\end{equation}
Moreover, the following variational inequality holds ($\tilde{\mathbb A}:=\tilde{\mathbb C}^{-1}$)
\begin{eqnarray}\label{Eqiv_Fitzpatrick_Ineq_Homo}
 &&\int_{\bQ}\int_\Omega\left\{ \left(\tilde{\mathbb A}[\omega]\sigma(x,\omega,t), \sigma(x,\omega,t)\right)+\left(\tilde{L}[\omega]z(x,\omega,t), z(x,\omega,t)\right)\right\}d{\cal P}dx
 \non\\
&&+\int_0^t\int_{\bQ}\int_\Omega f_{\tilde g}\left(\omega, B^T\sigma(x,\omega,\tau)-\tilde{L}[\omega]z(x,\omega,\tau), \partial_\tau z(x,\omega,\tau)\right)d{\cal P}dxd\tau\non\\
&& \le\int_{\bQ}\int_\Omega \left\{\left(\tilde{\mathbb A}[\omega]\sigma^{(0)}(x,\omega), \sigma^{(0)}(x,\omega)\right)+
 \left(\tilde{L}[\omega]\tilde{z}^{(0)}(x,\omega), \tilde{z}^{(0)}(x,\omega)\right)\right\}
d{\cal P}dx\non\\&&
+\int_0^t\int_{\bQ}\left(b(x,\tau), \partial_\tau u(x,\tau)\right)dxd\tau,
\end{eqnarray}
where $(v^{(0)}, v_1)\in H^1(\bQ, {\mathbb R}^3)\times
 L^2(\bQ, L^2_{pot}(\Omega))$ and $\sigma^{(0)}\in
L^2(\bQ, L^2_{sol}(\Omega))$ 
solve the linear elasticity problem
\begin{eqnarray}
- \di_x \sigma^{(0)}(x,\omega) &=&  b(x,0), \label{HoLePr1}\\
\sigma^{(0)}(x,\omega) &=& 
\tilde{\mathbb A}[\omega](\varepsilon(\na v^{(0)}(x)+v_1(x,\omega)) - B\tilde{z}^{(0)}(x,\omega)), 
  \label{HoLePr2}\\
v^{(0)}(x) &=& 0, \hspace{33ex} x \in \partial\bQ.\label{HoLePr3}
\end{eqnarray}
\end{Theor}
The careful reading of the proof of Theorem~\ref{ExResultPositiveSemiDef} suggests the
following result.
\begin{Prop}\label{uniform_estimates} 
Suppose that all assumptions of Theorem~\ref{ExResultPositiveSemiDef} are satisfied.
Then, the weak solution $(u_\eta,\sigma_\eta,z_\eta)$ of 
problem (\ref{microPr1}) - (\ref{microPr5}) (in the sense of Definition~\ref{WeakSol}) fullfills
the uniform estimates
\begin{eqnarray}
&&\{u_\eta\}_\eta \ {is}\  {uniformly}\  {bounded}\  {in}
 \ W^{1,q}(0,{T_e};L^q(\bQ; {\mathbb R}^3)), \non\\[1ex]
 &&\{\sigma_\eta\}_\eta\ {is}\  {uniformly}\  {bounded}\  {in}
 \ L^\infty(0,{T_e};L^{2}(\bQ; {\cal S}^3)),\non\\[1ex]
&&\{z_\eta\}_\eta \ {is}\  {uniformly}\  {bounded}\  {in}
 \ W^{1,q}(0,{T_e};L^q(\bQ; {\mathbb R}^N)), \non\\[1ex]
&&\{L^{1/2}z_\eta\}_\eta \ {is}\  {uniformly}\  {bounded}\  {in}
 \ L^\infty(0,{T_e};L^{2}(\bQ; {\mathbb R}^N)),\non\\[1ex]
 &&\{\Sigma_\eta\}_\eta \ {is}\  {uniformly}\  {bounded}\  {in}
 \ L^p(0, T_e; L^p(\bQ; {\mathbb R}^N)). \non
\end{eqnarray}
\end{Prop}
The result of Proposition~\ref{uniform_estimates} plays an important role in the proof of Theorem~\ref{HomogEquations} below.
\paragraph{Proof of Theorem~\ref{HomogEquations}}
\begin{proof}
Proposition~\ref{uniform_estimates}  provides the required uniform estimates for the solution
of the microscopic problem (\ref{microPr1}) - (\ref{microPr5}). Therefore,
due to Lemma~\ref{lem:weak-fA-conv-time} 
there exist functions $u_0$, $u_1$, $\sigma_0$ and $z_0$ with the prescribed
regularities in Theorem~\ref{HomogEquations} such that the convergence results in
\eq{WeakTwoScaleLimits} hold. We note that 
\eq{microPr2} gives equation \eq{HomogEqua3}, namely
\begin{eqnarray}\label{HelpEq0}
\sigma_0(x, \omega, t) = \tilde{\mathbb C}[\omega] \big( \ve (\na_x u_0(x, t)+ u_1(x, \omega, t)) 
  - Bz_0(x, \omega, t)\big), \ \ {\rm a.e.}
\end{eqnarray}
Next, we test equation \eq{microPr1} with
 a function $\phi\in C_0^\infty(\bQ,{\mathbb R}^3)$. Passing to the stochastic two-scale limit in the integral identity corresponding to
 \eq{microPr1} yields
\begin{eqnarray}\label{HelpEq1}
\int_{\bQ}\int_\Omega(\sigma_0(x,\omega,t),\ve(\na_x\phi(x)))dxd\cP(\omega)=
\int_{\bQ}(b(x,t),\phi(x))dx.
\end{eqnarray}
Now, we consider $\phi_{\tilde\omega,\eta}(x, t)=\eta\phi(x, t)V_\eta^{\tilde\omega}(x)$, where 
$\phi\in C_0^\infty(\bQ_{T_e}, {\mathbb R})$
and $\v \in L^q_{\rm pot}(\Omega)$ with potential $V_\eta^{\tilde\omega}$ given by Lemma 
\ref{lem:vanishing-ergodic-potential}, as another test function in  \eq{microPr1} and obtain
\begin{eqnarray}\label{HelpEq2}
&&\eta\int_0^{T_e}\int_{\bQ}\big(\sigma_\eta(x,t),V_\eta^{\tilde\omega}(x)\otimes\na_x\phi(x,t)\big)dx\,dt\non\\
&&+
\int_0^{T_e}\int_{\bQ}\big(\sigma_\eta(x,t),\phi(x,t)\ve(\na_x\v(\tau_{\frac{x}{\eta}}\tilde\omega))\big)dx\,dt
\\ && 
=
\int_0^{T_e}\int_{\bQ}(b(x,t),\phi_\eta(x,t))dx\,dt.\non
\end{eqnarray}
The stochastic two-scale limit in equation \eq{HelpEq2} yields
\begin{eqnarray}\label{HelpEq3}
\int_0^{T_e}\int_{\bQ}\int_\Omega\big(\sigma_0(x,\omega,t),\ve(\na_\omega\v(\omega))\big)\phi(x,t)d{\cal P}(\omega)dxdt=0.
\end{eqnarray}
Equation \eq{HelpEq3} implies that the integral identity
\begin{eqnarray}\label{HelpEq4}
\int_\Omega\big(\sigma_0(x,\omega,t),\ve(\na_\omega\v(\omega))\big)d{\cal P}(\omega)=0
\end{eqnarray}
holds for ever $\v\in L^q_{\rm pot}(\Omega)$ for a.e. $(x,t)\in {\bQ}\times(0,T_e).$ Integral equality \eq{HelpEq4} yields that 
$\sigma_0(x, \cdot,t)\in L^2_{\rm sol}(\Omega; {\cal S}^3)$ for a.e. $(x,t)\in {\bQ}\times(0,T_e)$.

To pass to the stochastic two-scale limit in the inequality
\begin{eqnarray}\label{Eqiv_Fitzpatrick_Ineq_two-scale}
 &&\int_{\bQ}\left\{ \left(\tilde{\mathbb A}[\tau_{\frac{x}{\eta}}\tilde\omega]\sigma_\eta(x,t), \sigma_\eta(x,t)\right)+\left(\tilde{L}[\tau_{\frac{x}{\eta}}\tilde\omega]z_\eta(x,t), z_\eta(x,t)\right)\right\}dx
 \non\\
&&+\int_0^t\int_{\bQ}f_{\tilde g}\left(\tau_{\frac{x}{\eta}}\tilde\omega, B^T\sigma_\eta(x,\tau)-\tilde{L}[\tau_{\frac{x}{\eta}}\tilde\omega]z_\eta(x,\tau), \partial_\tau z_\eta(x,\tau)\right)dx\,d\tau\non\\
&& \le\int_{\bQ} \left(\tilde{\mathbb A}[\tau_{\frac{x}{\eta}}\tilde\omega]\sigma^{(0)}_\eta(x), \sigma^{(0)}_\eta(x)\right)dx
+ \int_0^t\int_{\bQ}\left(b(x,\tau), \partial_\tau u_\eta(x,\tau)\right)dx\,d\tau\non
\\&&+
\int_{\bQ} \left(\tilde{L}[\tau_{\frac{x}{\eta}}\tilde\omega]\tilde{z}^{(0)}(x,\tau_{\frac{x}{\eta}}\tilde\omega), \tilde{z}^{(0)}(x,\tau_{\frac{x}{\eta}}\tilde\omega)\right)dx,\non
\end{eqnarray}
we use the results of Lemma~\ref{lem:weak-convergence-equivalent-norms} and Lemma~\ref{lem:General-Hom-Convex} and obtain
\begin{eqnarray}
 &&\int_{\bQ}\int_\Omega\left\{ \left(\tilde{\mathbb A}[\omega]\sigma(x,\omega,t), \sigma(x,\omega,t)\right)+\left(\tilde{L}[\omega]z(x,\omega,t), z(x,\omega,t)\right)\right\}d{\cal P}(\omega)dx
 \non\\
&&+\int_0^t\int_{\bQ}\int_\Omega f_{\tilde g}\left(\omega, B^T\sigma(x,\omega,\tau)-\tilde{L}[\omega]z(x,\omega,\tau), \partial_\tau z(x,\omega,\tau)\right)d{\cal P}(\omega)dxd\tau\non\\
&& \le\int_{\bQ}\int_\Omega \left\{\left(\tilde{\mathbb A}[\omega]\sigma^{(0)}(x,\omega), \sigma^{(0)}(x,\omega)\right)+
 \left(\tilde{L}[\omega]\tilde{z}^{(0)}(x,\omega), \tilde{z}^{(0)}(x,\omega)\right)\right\}
d{\cal P}(\omega)dx\non\\&&
+\int_0^t\int_{\bQ}\left(b(x,\tau), \partial_\tau u(x,\tau)\right)dxd\tau,\non
\end{eqnarray}
where $\sigma^{(0)}\in L^2(\bQ, L^2(\Omega))$ solves the following linear elasticity problem \eqref{HoLePr1}--\eqref{HoLePr3},
which is obtained by the passage to the stochastic two-scale limit in equations
\eq{LePr1} - \eq{LePr3}. Here, $v^{(0)}\in H^1(\bQ, {\mathbb R}^3)$ and 
$v_1\in L^2(\bQ, L^2_{\rm pot}(\Omega))$.

Therefore, we conclude that the limit function $(u_0, u_1, \sigma_0, z_0)$ satisfies the homogenized problem \eq{HomogEqua1} - \eq{HomogEqua6} and the
variational inequality \eq{Eqiv_Fitzpatrick_Ineq_Homo}.
\end{proof}


\section*{Appendix: Rothe's approximation functions}\label{Appendix}
Here we recall the definition of Rothe's approximation functions.
For any family $\{\xi^{n}_m\}_{n=0,...,2m}$ of functions in a reflexive Banach
space $X$ and for $h=T_e/m$, we define
{\it the piecewise affine interpolant} $\xi_m\in C([0,T_e],X)$ by
\begin{eqnarray}\label{RotheAffineinterpolant}
\xi_m(t):= \left(\frac{t}h-(n-1)\right)\xi^{n}_m+
\left(n-\frac{t}h\right)\xi^{n-1}_m \ \ {\rm for} \ (n-1)h\le t\le nh,
\end{eqnarray}
and {\it the piecewise constant interpolant} $\bar\xi_m\in L^\infty(0,T_e;X)$ by
\begin{eqnarray}\label{RotheConstantinterpolant}
\bar\xi_m(t):=\xi^{n}_m\  {\rm for} \ (n-1)h< t\le nh, \ 
 n=1,...,2^m, \ {\rm and} \ \bar\xi_m(0):=\xi^{0}_m.
\end{eqnarray}
For the further analysis we recall the following property of $\bar\xi_m$
and $\xi_m$: 
\begin{eqnarray}\label{RotheEstim}
\|\xi_m\|_{L^p(0,T_e;X)}\le\|\bar\xi_m\|_{L^p(-h,T_e;X)}\le
\left(h \|\xi^{0}_m\|^p_X+\|\bar\xi_m\|^p_{L^p(0,T_e;X)}\right)^{1/p},
\end{eqnarray}
where $\bar\xi_m$ is formally extended to $t\le0$ by $\xi^{0}_m$ and
$1\le p\le\infty$ (see \cite{Roubi05}).

\section*{Acknowledgement} M.H. is financed by DFG within research project SFB 1114 -- C05. N. S. acknowledges the financial support by the DFG within the research project NE 1498/4-1.

%


\bibliographystyle{plain} 
{\footnotesize
\bibliography{hom-convex-stochastic,literaturliste}

\begin{thebibliography}{10}

\bibitem{Alb98}
H.-D. Alber.
\newblock {\em Materials with {M}emory. {I}nitial-{B}oundary {V}alue {P}roblems
  for {C}onstitutive {E}quations with {I}nternal {V}ariables}, volume 1682 of
  {\em Lecture Notes in Mathematics}.
\newblock Springer, Berlin, 1998.

\bibitem{AlbNese09b}
H.-D. Alber and S.~Nesenenko.
\newblock Justification of homogenization in viscoplasticity: from convergence
  on two scales to an asymptotic solution in {$L^2(\Omega)$}.
\newblock {\em J. Multiscale Modeling}, 1(2):223--244, 2009.

\bibitem{allaire1992homogenization}
G.~Allaire.
\newblock Homogenization and two-scale convergence.
\newblock {\em SIAM Journal on Mathematical Analysis}, 23(6):1482--1518, 1992.

\bibitem{AubinFrankowska2008}
J.-P. Aubin and H.~Frankowska.
\newblock {\em Set-Valued Analysis}.
\newblock Modern Birkh\"auser Classics, Boston, 2008.

\bibitem{Barb76}
V.~Barbu.
\newblock {\em Nonlinear {S}emigroups and {D}ifferential {E}quations in
  {B}anach {S}paces}.
\newblock Editura Academiei, Bucharest, 1976.

\bibitem{Castaing77}
C.~Castaing and M.~Valadier.
\newblock {\em Convex analysis and {M}easurable {M}ultifunctions}, volume 580
  of {\em Lecture Notes in Mathematics Studies}.
\newblock Springer, Berlin, 1977.

\bibitem{Daley1988}
D.J. Daley and D.~Vere-Jones.
\newblock {\em An Introduction to the Theory of Point Processes}.
\newblock Springer-Verlag New York, 1988.

\bibitem{Damlamian07}
A.~Damlamian, N.~Meunier, and J.~Van Schaftingen.
\newblock Periodic homogenization of monotone multivalued operators.
\newblock {\em Nonlinear Anal., Theory Methods Appl.}, 67:3217--3239, 2007.

\bibitem{Doob1994}
J.L. Doob.
\newblock {\em Measure theory}.
\newblock Graduate {T}exts in {M}athematics. Springer, New York, 1994.

\bibitem{Fitzpatrick1988}
S.~Fitzpatrick.
\newblock Representing monotone operators by convex functions.
\newblock In {\em Workshop/Miniconference on Functional Analysis and
  Optimization (Canberra 1988)}, volume~20 of {\em Proceedings of the Centre
  for Mathematical Analysis}, pages 59--65. Australian National University,
  Canberra, Australia, 1988.

\bibitem{Francfort_Giacomini_2015}
G.~A. Francfort and A.Giacomini.
\newblock On periodic homogenization in perfect elasto-plasticity.
\newblock {\em Preprint}, 2015.

\bibitem{Giusti2003}
E.~Giusti.
\newblock {\em Direct {M}ethods in the {C}alculus of {V}ariations}.
\newblock World Scientific Publishing, New Jersey, 2003.

\bibitem{heida2014stochastic}
M.~Heida and B.~Schweizer.
\newblock Stochastic homogenization of plasticity equations.
\newblock 2014.

\bibitem{heida2016a}
Martin Heida.
\newblock Stochastic homogenization of rate-independent systems.
\newblock 2016.

\bibitem{Hu97}
Sh. Hu and N.~S. Papageorgiou.
\newblock {\em Handbook of {M}ultivalued {A}nalysis. {V}olume I: {T}heory}.
\newblock Mathematics and its {A}pplications. Kluwer, Dordrecht, 1997.

\bibitem{Mecke1967}
J.~Mecke.
\newblock {Station{\"a}re zuf{\"a}llige Ma{\ss}e auf lokalkompakten abelschen
  Gruppen}.
\newblock {\em Probability Theory and Related Fields}, 9(1):36--58, 1967.

\bibitem{Miel07}
A.~Mielke and A.M. Timofte.
\newblock Two-scale homogenization for evolutionary variational inequalities
  via the energetic formulation.
\newblock {\em SIAM J. Math. Anal.}, 39(2):642--668, 2007.

\bibitem{Nes07}
S.~Nesenenko.
\newblock Homogenization in viscoplasticity.
\newblock {\em SIAM J. Math. Anal.}, 39(1):236--262, 2007.

\bibitem{Nesenenko12a}
S.~Nesenenko.
\newblock Homogenization of rate-dependent inelastic models of monotone type.
\newblock {\em Asymptot. Anal.}, 81:1 -- 29, 2013.

\bibitem{NesenenkoNeff2012}
S.~Nesenenko and P.~Neff.
\newblock Well-posedness for dislocation based gradient visco-plasticity {II}:
  general non-associative monotone plastic flows.
\newblock {\em Mathematics Mechanics of Complex Systems}, 1(2):149--176, 2013.

\bibitem{Pankov97}
A.~Pankov.
\newblock {\em $G$-convergence and {H}omogenization of {N}onlinear {P}artial
  {D}ifferential {O}perators}.
\newblock Mathematics and its Applications. Kluwer, Dordrecht, 1997.

\bibitem{papanicolaou1979boundary}
G.~C. Papanicolaou and S.~R.~S. Varadhan.
\newblock Boundary value problems with rapidly oscillating random coefficients.
\newblock In {\em Random fields, {V}ol. {I}, {II} ({E}sztergom, 1979)},
  volume~27 of {\em Colloq. Math. Soc. J\'anos Bolyai}, pages 835--873.
  North-Holland, Amsterdam-New York, 1981.

\bibitem{Pas78}
D.~Pascali and S.~Sburlan.
\newblock {\em Nonlinear {M}appings of {M}onotone {T}ype}.
\newblock Editura Academiei, Bucharest, 1978.

\bibitem{Rockafellar68}
R.T. Rockafellar.
\newblock Integrals which are convex functionals.
\newblock {\em Pac. J. Math.}, 24:525--539, 1968.

\bibitem{Roubi05}
T.~Roubi$\v{c}$ek.
\newblock {\em Nonlinear {P}artial {D}ifferential {E}quations with
  {A}pplications}, volume 153 of {\em International Series of Numerical
  Mathematics}.
\newblock Birkh\"auser, Basel, 2005.

\bibitem{Schweizer10}
B.~Schweizer.
\newblock Homogenization of the {P}rager model in one-dimensional plasticity.
\newblock {\em Contin. Mech.Thermodyn.}, 20(8):459--477, 2009.

\bibitem{Schweizer_Veneroni_2014}
B.~Schweizer and M.~Veneroni.
\newblock Homogenization of plasticity equations with two-scale convergence
  methods.
\newblock {\em Appl. Anal.}, 2014.

\bibitem{Show97}
E.~Showalter.
\newblock {\em Monotone {O}perators in {B}anach {S}paces and {N}onlinear
  {P}artial {D}ifferential {E}quations}, volume~49 of {\em Mathematical Surveys
  and Monographs}.
\newblock AMS, Providence, 1997.

\bibitem{castaingvaladier1977convex}
M.~Valadier and C.~Castaing.
\newblock {\em Convex analysis and measurable multi-functions}.
\newblock Springer-Verlag, 1977.

\bibitem{Vis08b}
A.~Visintin.
\newblock Homogenization of nonlinear visco-elastic composites.
\newblock {\em J. Math. Pures Appl.}, 89(5):477--504, 2008.

\bibitem{Visintin08}
A.~Visintin.
\newblock Homogenization of the nonlinear {M}axwell model of viscoelasticity
  and of the {P}randtl-{R}euss model of elastoplasticity.
\newblock {\em Proc. Roy. Soc. Edinburgh Sec. A}, 138(6):1363--1401, 2008.

\bibitem{Visintin2013b}
Augusto Visintin.
\newblock Scale-transformations and homogenization of maximal monotone
  relations with applications.
\newblock {\em Asymptot. Anal.}, 82(3-4):233--270, 2013.

\bibitem{Zalinescu02}
C.~Zalinescu.
\newblock {\em Convex {A}nalysis in {G}eneral {V}ector {S}paces}.
\newblock World Scientific Publishing, New Jersey, 2002.

\bibitem{Zhikov_Pyatnitskii_2006}
V.~V. Zhikov and A.~L. Pyatnitskii.
\newblock Homogenization of random singular structures and random measures.
\newblock {\em Izvestiya RAN: Ser. Mat.}, 70(3):23--74, 2006.

\bibitem{zhikov2004two}
VV~Zhikov.
\newblock On two-scale convergence.
\newblock {\em Journal of Mathematical Sciences}, 120(3):1328--1352, 2004.

\end{thebibliography}
}

\end{document}